\pdfoutput=1
\documentclass[11pt]{amsart}
\usepackage[utf8]{inputenc}
\usepackage{amsmath}
\usepackage{amssymb}
\usepackage{amsthm}
\usepackage{tikz}
\usepackage{multirow}
\usepackage[all]{xy}
\usepackage[justification=centering]{caption}

\definecolor{vert}{RGB}{0,205,0}

\newtheorem{theorem}{Theorem}[section]
\newtheorem{definitio}[theorem]{Definition}
\newenvironment{definition}{\begin{definitio} \rm }{\end{definitio}}
\newtheorem{rem}[theorem]{Remark}
\newenvironment{remark}{\begin{rem} \rm }{\end{rem}}
\newtheorem{ex}[theorem]{Example}
\newenvironment{example}{\begin{ex} \rm }{\end{ex}}
\newtheorem{lemma}[theorem]{Lemma}
\newtheorem{proposition}[theorem]{Proposition}
\newtheorem{corollary}[theorem]{Corollary}
\newtheorem{question}[theorem]{Question}
\newtheorem*{theotd}{Theorem 3.2}
\newtheorem*{theost}{Theorem 7.3}

\newcommand{\Z}{\mathbb{Z}}

\newcommand{\R}{\mathbb{R}}

\newcommand{\RP}{\mathbb{R}\mathrm{P}}
\newcommand{\CP}{\mathbb{C}\mathrm{P}}
\newcommand{\MM}{\mathcal M}
\newcommand{\Int}{\mathrm{Int}}
\newcommand{\del}{\partial}

\newcommand{\id}{\mathrm{id}}
\newcommand{\cg}{[\![}
\newcommand{\cd}{]\!]}

\newcommand{\drawwr}[1]{\draw[white,line width=5pt,rounded corners=6pt] #1 \draw[red,rounded corners=6pt] #1}
\newcommand{\drawwb}[1]{\draw[white,line width=5pt,rounded corners=8pt] #1 \draw[blue,rounded corners=8pt] #1}
\newcommand{\drawwp}[1]{\draw[white,line width=5pt,rounded corners=8pt] #1 \draw[rounded corners=8pt,orange] #1}
\newcommand{\drawwo}[1]{\draw[white,line width=5pt,rounded corners=8pt] #1 \draw[rounded corners=8pt,purple] #1}
\newcommand{\draww}[1]{\draw[white,line width=5pt,rounded corners=8pt] #1 \draw[rounded corners=8pt] #1}
\newcommand{\buco}{(-1,0.1) ..controls +(0.5,-0.4) and +(-0.5,-0.4) .. (1,0.1) (-0.8,0) ..controls +(0.6,0.3) and +(-0.6,0.3) .. (0.8,0)}
\newcommand{\trou}{(-1,0) ..controls +(0.5,-0.4) and +(-0.5,-0.4) .. (1,0) (-0.8,-0.1) ..controls +(0.6,0.3) and +(-0.6,0.3) .. (0.8,-0.1)}

\title{Multisections of higher-dimensional manifolds}
\date{}

\author{Fathi Ben Aribi}
\author{Sylvain Courte}
\author{Marco Golla}
\author{Delphine Moussard}

\begin{document}

\begin{abstract}
 Generalizing Heegaard splittings of $3$--manifolds and Gay--Kirby trisections of $4$--manifolds, we consider multisections of higher-dimensional smooth (or PL) closed orientable manifolds, namely decompositions into $1$--handlebodies whose subcollections intersect along $1$--handlebodies, with global intersection a closed surface. With such a multisection one can associate a diagram. We prove that a multisection diagram determines a unique PL--manifold in all dimensions and a unique smooth manifold up to dimension $6$. Further, we show that any closed orientable smooth $5$--manifold admits a multisection.
\end{abstract}

\maketitle

\section{Introduction}

While Heegaard splittings of $3$--manifolds have been widely studied, their $4$--dimensional analogues, the so-called trisections, have been recently introduced by Gay and Kirby \cite{GayKirby} and constitute an active and prolific subject of investigation. Beyond Heegaard splittings and trisections, it is natural to ask if higher-dimensional manifolds also admit analogous decompositions. Rubinstein and Tillmann introduced in \cite{RuTi} a notion of multisection for PL--manifolds, which coincides with Heegaard splittings in dimension $3$ and trisections in dimension $4$. Nevertheless the analogy is not as strong as one could wish; in particular, these multisections have no diagrammatic representations.

We introduce here an alternative notion of multisection. We call a multisection of a smooth (or~PL) $(n+1)$--manifold a decomposition into $n$ $1$--handlebodies such that each subcollection intersects along a $1$--handlebody except the global intersection which is a closed surface --- the central surface of the multisection. A diagram of such a multisection is a family of cut-systems on the central surface defining the $3$--dimensional $1$--handlebodies of the decomposition. Building on results of Laudenbach--Poénaru \cite{LP}, Montesinos \cite{Mon} and Cavicchioli--Hegenbarth \cite{CH}, we prove that a multisection diagram determines a unique PL--manifold, and a unique smooth manifold up to dimension $6$. Starting in dimension $7$, there exist distinct smooth manifolds with diffeomorphic multisection diagrams, and from dimension $8$, there may exist multisection diagrams not realizable by smooth manifolds.

\begin{theotd}
 Let $n\geq2$. Every abstract $n$--section diagram is the diagram of some multisected PL $(n+1)$--manifold which is unique up to multisection-preserving PL--homeo\-mor\-phism. Further, if $n\leq 6$, every abstract $n$--section diagram is the diagram of some multisected smooth $(n+1)$--manifold, which, for $n\leq 5$, is unique up to multisection-preserving diffeomorphism.
\end{theotd}

Heegaard splittings and trisections exist for all smooth manifolds in dimension $3$ and $4$ respectively. Using Morse theory, we push this existence result one dimension higher.

\begin{theost}
 Every closed connected orientable smooth $5$--manifold admits a quadrisection.
\end{theost}

\begin{remark}
 It is known that all PL--manifolds of dimension~$5$ are smoothable, so the above theorem remains true for PL $5$--manifolds.
\end{remark}

\begin{question}
 Does every closed connected orientable smooth manifold of dimension greater than $5$ admit a multisection?
\end{question}

Heegaard splittings and trisections of a given manifold are known to be unique up to some stabilization moves. We introduce higher-dimensional stabilization moves, which can be defined as connected sums of a multisected manifold with a genus--$1$ multisection of the standard sphere. We also describe them as ambient cut-and-paste operations on the multisected manifold.

\begin{question}
 For a closed connected orientable smooth manifold $W$ of dimension greater or equal to $5$, do every two multisections of $W$ become isotopic after some number of stabilizations?
\end{question}

Since a multisection diagram determines a manifold up to homeomorphism, it encodes the topological invariants of the manifold. We explain how to compute the homology of a multisected manifold from a multisection diagram. More precisely, we show that the homology of the manifold is the homology of a complex defined using the first homology group of the central surface and its subgroups generated by the curves of the diagram.

\subsection*{Plan of the paper}
In Section~\ref{sec:multisections}, we give the preliminary definitions. Section~\ref{sec:diagrams} treats the question of diagrams and contains the proof of Theorem~\ref{thm:diagramVSmultisection}. In Section~\ref{sec:examples}, we give some examples; we exhibit in particular all genus--$0$ and genus--$1$ PL--multisections. Section~\ref{sec:stabilizations} is devoted to stabilization moves
on multisections. In Section~\ref{sec:4Dquadri}, we make explicit a relation between Islambouli--Naylor $4$--dimensional multisections and our $5$--dimensional quadrisections, and apply it to the study of $4$--dimensional quadrisections up to stabilization moves. Section~\ref{sec:5Dexistence} gives the existence proof of multisections in dimension~$5$ (Theorem~\ref{thm:existquadri}) and Section~\ref{sec:moreExs} illustrates the proof with some more examples. Finally, in Section~\ref{sec:homology}, we explain how to compute the homology of a manifold from a multisection diagram (Theorem~\ref{th:homology}).

\subsection*{Acknowledgements}
We would like to thank Benjamin Audoux, Rudy Dissler, David Gay and Jean-Paul Mohsen for helpful discussions.
This work was supported by the ANR project SyTriQ ANR-20-CE40-0004-01.
The first author was supported by the FNRS in his ``Research Fellow'' position at UCLouvain, under Grant no. 1B03320F.

\section{Multisections}
\label{sec:multisections}

We call a \emph{$1$--handlebody} (or simply {\em handlebody}) a compact connected orientable smooth manifold that can be constructed using only $0$-- and $1$--handles, {\em ie} which admits a Morse function (constant on the boundary and attaining its maximum there) with only critical points of index $0$ and $1$. Equivalently, this is a manifold which is diffeomorphic to a boundary connected sum of $k$ copies of $S^1\times D^n$ for some $k\geq0$, $n>1$. The number $k$ is called the \emph{genus} of the $1$--handlebody.

\begin{definition}\label{def:multisection}
For $n\geq2$, an \textit{$n$--section of genus $g$} of a closed smooth $(n+1)$--manifold $W$ is a decomposition
$W=\cup_{i=1}^n W_i$ satisfying the following conditions, where, for $I$ a non-empty subset of $\{1,\dots,n\}$, we set $W_I=\cap_{i\in I}W_i$:
\begin{itemize}
 \item $W_I$ is a submanifold with corners whose codimension--$k$ stratum is the union of $\mathrm{Int}(W_J)$ for $J$ containing $I$ such that $|J|=|I|+k$,
 \item $W_{\{1,\dots,n\}}$ is diffeomorphic to a closed orientable genus--$g$ surface,
 \item for $I\neq \{1,\dots,n\}$, after smoothing the corners, $W_I$ is diffeomorphic to a $1$--handlebody of dimension $n+2-|I|$.
\end{itemize}
The {\em $k$--spine} of the $n$--section is the union $\cup_{|I|=n-k+2}W_I$ of its $k$--dimensional pieces.
An $n$--section of $W$ will typically be denoted by a single letter $\mathcal M$.
We call \emph{multisection} an $n$--section for some~$n$.
The \emph{multisection genus} of a closed manifold $W$ is the minimal $g$ among all multisections of~$W$.
\end{definition}

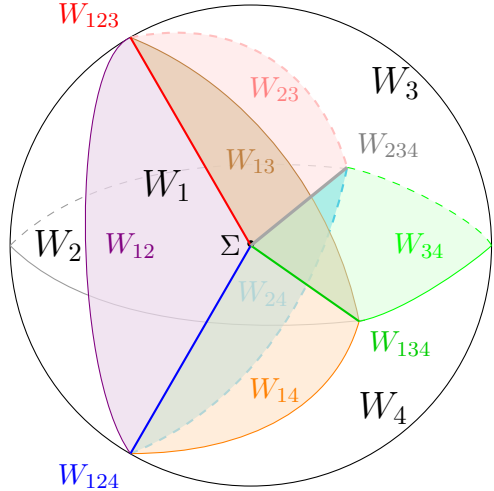
\begin{figure} [htb]
\begin{center}
\begin{tikzpicture} [scale=1.6]
\draw (0,0) circle (2);

\draw[color=black!40] (-2,0) .. controls +(1,-1) and +(.4,.08) .. (.9,-.63);
\draw[color=black!40, dashed] (-2,0) .. controls +(1,1) and +(.4,-.08) .. (.75,.65);

\fill[color=cyan!80, fill opacity=0.4] (0,0) -- (-1,-1.73) .. controls +(1,0.4) and +(-0.2,-1) .. (0.8,0.65); -- cycle;
\draw[color=cyan, dashed, thick] (-1,-1.73) .. controls +(1,0.4) and +(-0.2,-1) .. (0.8,0.65);
\draw (.1,-.4) node  {\color{cyan} $W_{24}$};

\fill[color=pink!40, fill opacity=0.7] (0,0) -- (-1,1.73) .. controls +(1,0.2) and +(-0.2,0.7) .. (0.8,0.65) -- cycle;
\draw[color=pink,dashed,thick] (-1,1.73) .. controls +(1,0.2) and +(-0.2,0.7) .. (0.8,0.65);
\draw (.2,1.3) node  {\color{pink!250} $W_{23}$};

\fill[color=violet!15, fill opacity=0.7] (0,0) -- (-1,1.73) .. controls +(-.5,-.2) and +(-.5,0.7) .. (-1,-1.73) -- cycle;
\draw[color=violet] (-1,1.73) .. controls +(-.5,-.2) and +(-.5,0.7) .. (-1,-1.73);
\draw (-1,0) node  {\color{violet} $W_{12}$};

\fill[color=orange!20, fill opacity=0.7] (0,0) -- (-1,-1.73) .. controls +(1,-0) and +(-0.2,-0.7) .. (0.9,-0.63) -- cycle;
\draw[color=orange] (-1,-1.73) .. controls +(1,-0) and +(-0.2,-0.7) .. (0.9,-0.63);
\draw (.2,-1.2) node  {\color{orange} $W_{14}$};

\fill[color=brown!40, fill opacity=0.8] (0,0) -- (-1,1.73) .. controls +(1,-0.4) and +(-0.2,1) .. (0.9,-0.63) -- cycle;
\draw[color=brown] (-1,1.73) .. controls +(1,-0.4) and +(-0.2,1) .. (0.9,-0.63);
\draw (0,.7) node  {\color{brown} $W_{13}$};

\fill[color=green!20, fill opacity=0.4] (0,0) -- (.9,-.63) .. controls +(.4,.08) and +(-.1,-.1) .. (2,0)
.. controls +(-.1,.2) and +(.4,-.08) .. (.8,.65) -- cycle;
\draw[color=green] (.9,-.63) .. controls +(.4,.08) and +(-.1,-.1) .. (2,0);
\draw[color=green, dashed] (2,0) .. controls +(-.1,.2) and +(.4,-.08) .. (.8,.65);
\draw (1.4,0) node  {\color{green} $W_{34}$};

\draw (0,0) node{\Huge $\cdot$};
\draw (0,0) node [left] {$\Sigma$};
\draw[color=red,thick] (0,0) -- (-1,1.73);
\draw (-1,1.73) node [above left] {\color{red} $W_{123}$};
\draw[color=blue,thick] (-1,-1.73) -- (0,0);
\draw (-1,-1.73) node [below left] {\color{blue} $W_{124}$};
\draw[color=vert,thick]  (0.9,-0.63) -- (0,0);
\draw (.9,-.63) node [below right] {\color{vert} $W_{134}$};
\draw [color=gray!70,very thick] (0,0) -- (0.8,0.65);
\draw (.8,.65) node [above right] {\color{gray} $W_{234}$};
\draw (-.7,.5) node  {\Large $W_{1}$};
\draw (-1.6,0) node  {\Large $W_{2}$};
\draw (1.2,1.35) node  {\Large $W_{3}$};
\draw (1.1,-1.35) node  {\Large $W_{4}$};
\end{tikzpicture}
\end{center} \caption{A schematic $4$--section of a smooth $5$--manifold\\{\footnotesize The four rays represent the 3--dimensional pieces $W_{ijk}$ intersecting along the central surface $\Sigma=W_{1234}$. They define six 2--dimensional sectors, corresponding to the 4--dimensional pieces $W_{ij}$, that divide the 5--dimensional whole into four pieces $W_i$.}}
\label{fig:quadrisection}
\end{figure}

A $2$--section of a $3$--manifold is a \emph{Heegaard splitting} and a $3$--section of a $4$--manifold is a \emph{trisection} in the sense of Gay--Kirby \cite{GayKirby}. A schematic of a $4$--section (or {\em quadrisection}) is represented in Figure~\ref{fig:quadrisection}.

The definition enjoys the following inductive property: for each $I$, $\partial W_I$ is a smooth manifold after smoothing corners and
inherits an $(n-|I|)$--section given by $\del W_I=\cup_{j\notin I} (W_I\cap W_j)$.

In what follows, to simplify the notation, we shall write $W_{i_1\dots i_k}$ for $W_{\{i_1,\dots,i_k\}}$.



\begin{remark}
 Definition~\ref{def:multisection} extends to the PL setting by replacing smooth submanifolds by PL--submanifolds and diffeomorphism by PL--homeomorphism everywhere. In this case, the $W_I$ are PL--submanifolds of $W$ which are PL--homeomorphic to PL--handlebodies.
\end{remark}

The requirement that the $W_I$ are submanifolds with corners and the condition on the strata impose a local model at the intersection of the different pieces of the decomposition. To make it explicit, we first define standard decompositions of simplices.

\begin{definition}\label{def:simplex}
For $k>0$, let $\Delta^{k-1}$ be a simplex of dimension $k-1$ with vertices $p_1,\dots,p_k$. Let $0$ be the barycenter of the simplex. For $1\leq i\leq k$, we define $\Delta^{k-1}_i$ to be the convex hull of $(0,p_1,\dots,\widehat{p_i},\dots,p_k)$, where $\widehat{p_i}$ means we omit the term $p_i$.
We refer to the decomposition $\Delta^{k-1}=\cup_i \Delta_i^{k-1}$ as the \emph{standard decomposition} of $\Delta^{k-1}$ (see Figure \ref{fig:simplex}).
\end{definition}

\begin{figure} [htb]
	\begin{center}
		\begin{tikzpicture}

\begin{scope}[xshift=0cm,yshift=0cm,scale=1.5]
\draw [color=blue, thick] (0,-1)--(0,0);
\draw [color=blue, thick] (0,1)--(0,0);

\draw (0,0) node{\huge $\cdot$};
\draw (0,1) node{\huge $\cdot$};
\draw (0,-1) node{\huge $\cdot$};

\draw (0,-1.7) node{$\Delta^1$};

\draw (0,0) [right] node{\small $0$};
\draw (0,1) [right]  node{$p_1$};
\draw (0,-1) [right]  node{$p_2$};

\draw (0,1/2) [left]  node{\color{blue}$\Delta_2^1$};
\draw (0,-1/2) [left]  node{\color{blue}$\Delta_1^1$};
\end{scope}

\begin{scope}[xshift=6cm,yshift=-0.3cm,scale=1.5]

\draw  (0,1)--(-1.73/2,-.5)--(1.73/2,-.5)--(0,1);

\draw  (0,0)--(0,1);
\draw (0,0)--(-1.73/2,-.5);
\draw  (0,0)--(1.73/2,-.5);

\draw (0,0) node{\huge $\cdot$};
\draw (0,1) node{\huge $\cdot$};
\draw (-1.73/2,-.5) node{\huge $\cdot$};
\draw (1.73/2,-.5) node{\huge $\cdot$};

\draw (0,-1.7) node{$\Delta^2$};

\draw (0.1,0.1)  node{\small $0$};
\draw (0,1) [above]  node{$p_1$};
\draw (-1.73/2,-.5) [left] node{$p_2$};
\draw (1.73/2,-.5) [right] node{$p_3$};

\draw (.65/2,.05)   node{\color{blue}$\Delta_2^2$};
\draw (-.65/2,.05)   node{\color{blue}$\Delta_3^2$};
\draw (0,-.3)   node{\color{blue}$\Delta_1^2$};
\end{scope}

\begin{scope}[xshift=12cm,yshift=0cm,scale=1.5]

\draw (0,0) node{\huge $\cdot$};
\draw (0,1) node{\huge $\cdot$};
\draw (1,-.2) node{\huge $\cdot$};
\draw (-1,-.2) node{\huge $\cdot$};
\draw (.5,-1) node{\huge $\cdot$};

\draw[color=gray] (0,0)--(0,1);
\draw[color=gray] (0,0)--(1,-.2);
\draw[color=gray] (0,0)--(-1,-.2);
\draw[color=gray] (0,0)--(.5,-1);

\draw (0,1)--(1,-.2)--(.5,-1)--(-1,-.2)--(0,1);
\draw[dashed] (1,-.2)--(-1,-.2);
\draw (0,1)--(.5,-1);

\draw (0,-1.7) node{$\Delta^3$};

\draw (0.1,-.4)   node{\color{blue!60}$\Delta_1^3$};
\draw (.5,0)   node{\color{blue}$\Delta_2^3$};
\draw (-.1,.4)   node{\color{blue!60}$\Delta_3^3$};
\draw (-.5,0)   node{\color{blue}$\Delta_4^3$};

\draw (0,0) [right]  node{\color{black!60}\small $0$};
\draw (0,1) [above]  node{$p_1$};
\draw (-1,-.2) [left] node{$p_2$};
\draw (.5,-1) [below]  node{$p_3$};
\draw (1,-.2) [right] node{$p_4$};

\end{scope}
\end{tikzpicture}
\end{center} \caption{Standard decompositions of simplices}
\label{fig:simplex}
\end{figure}
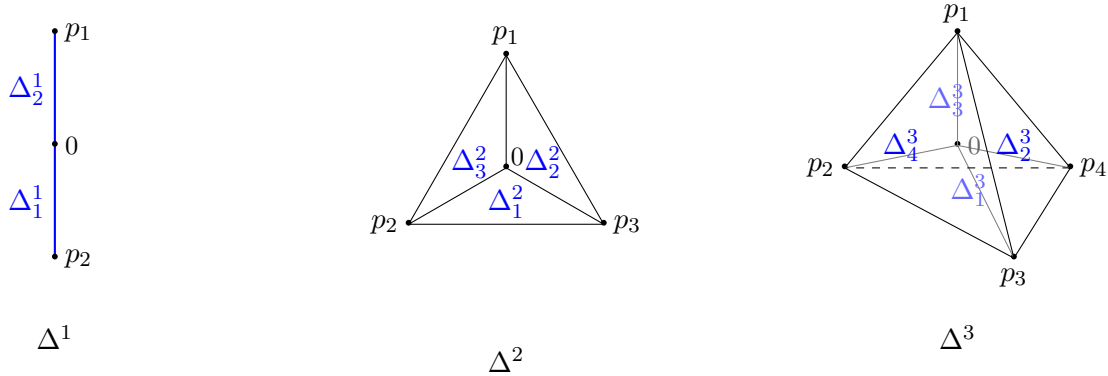


In the next lemma, we use the notation $\Delta^I$ to denote the simplex $\Delta^{|I|-1}$ where the set of indices $\{1,\dots,|I|\}$ is replaced by $I$.

\begin{lemma}
For each $I\neq \emptyset$ there exists a codimension $0$ embedding $\varphi:W_I\times \Delta^I \to W$ such that $\varphi(x,0)=x$ for all $x\in W_I$ and $\varphi^{-1}(W_J)=W_I\times \Delta^I_J$ for all $J \subset I$. In particular each $W_I$ has a trivial normal bundle in $W$.
\end{lemma}
\begin{proof}
Since $W_I\subset\del W_{I\setminus\{i\}}$ for all $i\in I$, we can find a vector field $v_i$ along $W_I$ which is tangent to $W_{I\setminus \{i\}}$
and inward pointing inside $W_{I\setminus\{i\}}$. Moreover we can assume that $\sum_{i\in I} v_i=0$ (for instance if $I=\{1,2,3\}$, we see that $v_3$ can be chosen in $\{av_1+bv_2, a<0, b<0\}$
because $W_1$, $W_2$ and $W_3$ all have a corner along $W_{123}$). The tubular neighborhood theorem then provides
an embedding $\varphi:W_I\times \Delta^I \to W$ such that $\varphi(x,0)=x$ and $d_{(x,0)}\varphi(0,p_i)=v_i(x)$ (where $p_i$ is the vertex indexed $i$ in
the standard simplex $\Delta^I$ where $\sum_{i\in I} p_i=0$). The decomposition of $W_I\times \Delta^I$ induced by $\varphi$
does not yet agree with the one induced by the standard decomposition of $\Delta^I$, but by construction their tangent planes agree along $W_I\times \{0\}$,
namely $d_{(x,0)}\varphi(T_xW_I\times T_0\Delta^I_J)=T_x W_J$ for all $J\subset I$. We view this as a family of
decompositions of $\Delta^I$ parametrized by $x\in W_I$ which all agree up to order $1$ at the origin. 
We claim that
these decompositions can be straightened by a family of (germs of) diffeomorphisms of $\Delta^I$ preserving $0$, and the remainder of this proof consists in proving this claim.

We first explain
the process in the $2$-dimensional case. Consider $\R^2$ with three half-lines $X_{13}=\{x=y, x\leq 0\}$, $X_{23}=\{x=0, y\geq 0\}$ and $X_{12}=\{y=0, x\geq 0\}$.
If $X'_{12}$ is another half-line tangent to $X_{12}$ at the origin, it can be written as $\{y=g(x), x\geq 0\}$ where $g:\R\to\R$ is a smooth function
with $g(0)=0$ and $g'(0)=0$. We then define
$\psi(x,y)=(x-g(x), y-g(x))$, this is a smooth local diffeomorphism preserving $X_{13}$, $X_{23}$ and mapping $X'_{12}$ to $X_{12}$. This process works parametrically
and therefore allows to straighten the family of decompositions smoothly along $W_{123}$ one edge at a time. Note that the three edges can be permuted by linear automorphisms, so that it is enough to be able to straighten one of them.

In higher dimension, we similarly straighten all faces of the simplex by analogous diffeomorphisms, starting from edges all the way up to the facets.
Explicitly, in $\R^n$, we consider the standard decomposition in $(n+1)$ pieces spanned by the vectors $(1,0,\dots,0)$, $(0,1,\dots,0)$, ...,$(0,,\dots,0,1)$ and $(-1,\dots,-1)$, namely
the $k$-faces consist of all linear combination with positive coefficients of $k$ of these vectors. Assume we have straigthened all faces up to dimension $p-1$, and we wish to straighten the
$p$--face $F$ which is tangent to $x_{p+1}=\dots=x_n=0$. Note that, like in dimension $2$, up to applying linear automorphisms, we only need to show how to straighten one of the $p$--faces.
The $p$--face $F$ can be written (near the origin) as a graph $x_{p+1}=f_{p+1}(x_1,\dots,x_p), \dots, x_n=f_n(x_1,\dots,x_p)$.
We will progressively achieve $f_{p+1}=0$, $f_{p+2}=0$, ..., $f_n=0$ by successively applying certain diffeomorphisms. Assume
we have already achieved $f_{p+1}=\dots=f_k=0$, so $F$ is given by
\[F=\{(x_1,\dots,x_p,0,\dots,0,f_{k+1}(x_1,\dots,x_p),\dots,f_n(x_1,\dots,x_p))\}\]
where automatically $f_{k+1}$ factors as $f_{k+1}(x_1,\dots,x_p)=x_1\dots x_p\,a(x_1,\dots,x_p)$
because we have already straightened the lower dimensional faces.

For $I$ a subset of $\{1,\dots,n\}$, we denote $S_m^\alpha(I)=\sum_{i_1<\dots<i_m\in I}x_{i_1}^\alpha\dots x_{i_m}^\alpha$ and define a diffeomorphism $
x=(x_1,\ldots,x_n) \mapsto 
\varphi(x)=(X_1(x),\dots,X_n(x))$ by setting:

for $i\leq p$ or $i>k$ :
$$X_i=x_i-\left(S_p^1(\cg1,p\cd\big)+\sum_{q=1}^{\lfloor\frac p2\rfloor} x_i^q S^1_{p-2q}\big(\cg1,p\cd\big) S_1^q\big(\cg p+1,k\cd\big)\right) a(x_1,\ldots,x_p),$$

and for $p<i\leq k$ :
$$X_i=x_i-\left(\sum_{q=1}^{\lfloor\frac p2\rfloor} x_i^q S^1_{p-2q}\big(\cg1,p\cd\big) S_1^q\big(\cg p+1,k\cd\setminus\{i\}\big)+\sum_{q=0}^{\lfloor\frac{p-1}2\rfloor} x_i^{2q+1} S^1_{p-2q-1}\big(\cg1,p\cd\big)\right) a(x_1,\ldots,x_p).$$

Remark that the other $p$-faces are of the form $\{x_i=0, i\in I\}$ for $I\subset \{1,\dots,n\}$, $|I|=n-p$, $I\neq \{p+1,\dots,n\}$ 
or of the form $\{x_i=x_j, i,j\in I\}$ with $I\subset\{1,\dots,n\}$, $|I|=n-p+1$. 
One can check that, under the diffeomorphism $\varphi$,

\begin{itemize}
 \item the face $F$ is mapped to a face of the form 
 \[(x_1,\dots,x_p, 0,\dots, 0 , g_{k+2}(x_1,\dots,x_p),\dots,g_n(x_1,\dots,x_p))\]
  for some functions $g_{k+2},\dots,g_n$,
 \item every other $p$--face of the form $\{x_i=0, i\in I\}$ for $I\subset \{1,\dots,n\}$, $|I|=n-p$, $I\neq \{p+1,\dots,n\}$ is mapped to itself,
 \item every other $p$--face of the form $\{x_i=x_j, i,j\in I\}$ with $I\subset\{1,\dots,n\}$, $|I|=n-p+1$), is mapped to itself.
\end{itemize}

The lower-dimensional faces are automatically preserved since they are intersections of $p$-faces.
After a finite number of iterations, we have straightened the face $F$ and we continue the process to straighten all faces.
\end{proof}

\begin{example}\label{ex:sphere}
For $n\geq 1$, let $S^{n+1}$ be the unit sphere of $\R^{n+2}$ and $p\colon S^{n+1}\to \R^{n-1}$
the projection map forgetting the last three coordinates. The standard decomposition of an $(n-1)$--simplex induces a decomposition of $\R^{n-1}$ into $n$ pieces centered at the origin (identifying $\R^{n-1}$ with the interior of the simplex). Pull it back by $p$: this is a multisection
of genus $0$ of $S^{n+1}$. The central surface is a 2--sphere and all other pieces are balls, which are $1$--handlebodies of genus~$0$, of varying dimensions.
\end{example}

Beware that if we use non-standard decompositions, like on the middle and on the right of Figure~\ref{fig:nondiff2disk},
we obtain very similar-looking decompositions of a manifold, but which are not strictly speaking multisections as they do not
abide by the above local model up to \emph{diffeomorphism}.

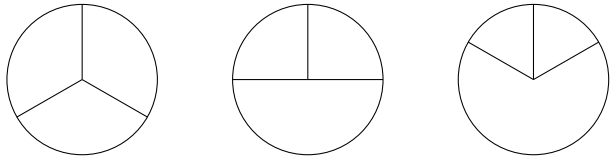
\begin{figure}[htb]
\begin{center}
\begin{tikzpicture}
 \foreach \x in {0,3,6}
 \draw (\x,0) circle (1);
 \foreach \x/\y in {0/1,0.87/-0.5,-0.87/-0.5}
 \draw (0,0) -- (\x,\y);
 \foreach \x/\y in {3/1,2/0,4/0}
 \draw (3,0) -- (\x,\y);
 \foreach \x/\y in {6/1,6.87/0.5,5.13/0.5}
 \draw (6,0) -- (\x,\y);
\end{tikzpicture}
\caption{Non-diffeomorphic decompositions of the $2$--disk} \label{fig:nondiff2disk}
\end{center}
\end{figure}

\begin{remark}\label{rem:orient}
If a closed $(n+1)$--manifold $W$ admits a multisection, then it is orientable. Indeed, an orientation
of the central surface $\Sigma$, together with a total order on the set of codimension--$0$ pieces, induce
an orientation of $W$ and of each piece $W_I$ inductively via the following rule:
\[\del W_I = \bigcup_{j\not\in I}(-1)^{|\{i\in I\mid i<j\}|}W_{I\cup\{j\}}\]
and the ``outward normal first" convention.

The normal bundle of $\Sigma$ in $W$ has a trivialization defined by choosing a basis $(v_1,\dots,v_{n-1})$
of $TW/T\Sigma$ such that $v_i\in T W_{\{i\}^c}$ and $v_i$ is outward pointing from $W_{\{i\}^c}$. Given an orientation
of $\Sigma$ we can check that the orientation induced by the above condition is compatible
with the orientation near $\Sigma$ induced by this trivialization.

Finally, if we reverse the orientation of $\Sigma$, the induced orientation of $W$ is also reversed.
If we change the order on the indexing set by a permutation $\sigma$, the orientation of $W$ is
reversed if and only if $\sigma$ is odd.
\end{remark}

\begin{remark}
The connected sum of multisected manifolds is again multisected (by performing the connected sum near
a point of the central surface).
\end{remark}

\section{Multisection diagrams}
\label{sec:diagrams}

A {\em complete system of disks} of a $3$--dimensional handlebody $H$ is a collection of disjoint properly embedded disks in $H$ such that cutting $H$ along all these disks yields a $3$--sphere.
A {\em cut-system} for $H$ is an unordered family of disjoint simple closed curves on~$\partial H$ that bound a complete system of disks of $H$.
Given an $n$--section $\mathcal M$ of a smooth $(n+1)$--manifold $W$ as above, denote $\Sigma=\cap_{1\leq i\leq n}W_i$ the central surface and choose for all $i\in\{1,\dots,n\}$ a cut-system $\alpha^i=(\alpha_j^i)_{1\leq j\leq g}$ on~$\Sigma$ for the 3--dimensional handlebody $\cap_{k\neq i}W_k$. Then $(\Sigma;\alpha^1,\dots,\alpha^n)$ is an {\em $n$--section diagram} for $(W,\mathcal M)$.
This is not unique, but it is well-known that each system of curves $(\alpha^i_1,\dots,\alpha^i_g)$ is unique up to permutation and handleslides. Hence the $n$--section diagram associated to a multisected manifold is unique up to diffeomorphism and handleslides (performed independently in each family $\alpha^i$).
See Figures~\ref{fig:gen1quad}, \ref{fig:genusn} and \ref{fig:S2xS3} for examples of multisection diagrams.

We define an abstract multisection diagram recursively, as follows.
\begin{definition}\label{def:diagram}
An {\em abstract $n$--section diagram} is a genus--$g$ closed surface $\Sigma$ with $n$ families of $g$ disjoint and homologically independent simple closed curves, such that any subcollection of $k$ of these families, with $2\leq k\leq n-1$, is a $k$--section diagram for a multisection of a connected sum of copies of $S^1\times S^k$.
\end{definition}

\begin{theorem}\label{thm:diagramVSmultisection}
Let $n\geq2$.
\begin{enumerate}
\item For $n\leq 6$, every abstract $n$--section diagram is the diagram of some smooth multisected $(n+1)$--manifold.
\item For $n\leq 5$, the multisected manifold associated to a given diagram is unique up to multisection-preserving diffeomorphism.
\item For arbitrary $n$, every abstract $n$--section diagram is the diagram of some PL--multisected $(n+1)$--manifold
which is unique up to multisection-preserving PL--homeomorphism.
\end{enumerate}
\end{theorem}

Before proving Theorem \ref{thm:diagramVSmultisection}, let us state some remarks.

\begin{remark}\label{rem:exotic}
Milnor has shown that there exist exotic $7$--spheres~\cite{MilnorS7}. Take such a sphere, {\em ie} a smooth manifold $\Sigma^7$ which is PL--homeomorphic but not diffeomorphic to $S^7$.
Starting from the standard multisection of genus $0$ of $S^7$ from Example~\ref{ex:sphere}, we can perform a connected sum of $S^7$ with $\Sigma^7$ inside of one
of the $7$--dimensional pieces, say $W_1$. The modified $W_1$ is still contractible and a direct corollary of the h-cobordism theorem shows that
this manifold is still diffeomorphic to a ball. We have thus obtained a multisection of genus $0$ of $\Sigma^7$.
So $S^7$ and $\Sigma^7$ share a common multisection diagram, namely $S^2$ with an empty system of curves.
Hence, for $n\geq 6$, we cannot hope for the uniqueness up to  diffeomorphism of a manifold associated with an abstract $n$--section diagram.
\end{remark}

\begin{remark}
According to smoothing theory, all PL--manifolds of dimension at most $7$ are smoothable. Starting in dimension $8$, there are PL--manifolds which cannot be smoothed, and thus there may exist an $n$--section diagram which corresponds to a PL--manifold but to no smooth manifold (for that we would need
to know that such a manifold admits a PL--multisection, or even better show that all PL--manifolds do).
\end{remark}

The collection of curves in an abstract $n$--section diagram tells us how to glue the $3$--dimensional $1$--handlebodies $\cap_{k\neq i} W_k$
to the central surface $\Sigma=\cap_k W_k$. However we have no information on how to glue the higher-dimensional pieces, leading for instance
to the ambiguity explained in Remark~\ref{rem:exotic}. Fortunately, Laudenbach--Poénaru \cite{LP} have shown that any diffeomorphism
of the boundary of a $4$--dimensional $1$--handlebody $H$ extends to a diffeomorphism of $H$. This proves that there is only one way to
glue the $4$--dimensional pieces in our setting. In higher dimension, we have the following result of Cavicchioli--Hegenbarth \cite[Proposition~3.2]{CH},
which relies on surgery theory.

\begin{theorem}[Cavicchioli--Hegenbarth]\label{thm:cav:heg}
 If $H$ is an $n$--dimensional $1$--handlebody with $n\geq 5$, any PL--homeomorphism of $\partial H$ extends to a PL--homeomorphism of $H$.
This also holds for diffeomorphisms if $n=5$ or $6$.
\end{theorem}

\begin{remark}
For $n=4$, this also holds for diffeomorphisms (Laudenbach--Poénaru \cite{LP}) and for PL--homeomorphisms (Montesinos \cite{Mon}).
For $n=3$, this holds if we further require the diffeomorphism or PL--homeomorphism of the surface to preserve a cut-system
$(\alpha_1,\dots,\alpha_g)$ (since any diffeomorphism or PL--homeomorphism of $S^2$ extends to $D^3$).
\end{remark}

\begin{proof}[Proof of Theorem~\ref{thm:diagramVSmultisection}]
We prove by induction on $n$ the existence part in dimension $\leq n+1$ and uniqueness in dimension $\leq n$.
In the differentiable category, the proof stops at $n=6$, while in the PL category it goes on for arbitrary $n$.

Pick an abstract $n$--section
diagram $(\Sigma;\alpha^1,\dots,\alpha^g)$.

We start the construction of a multisected manifold associated with this diagram by considering
$\Sigma\times \Delta$ where $\Delta=\Delta^{n-1}$ has the standard decomposition as in Definition~\ref{def:simplex}.
Pick a $1$--handlebody $H$ of genus $g$ and diffeomorphisms $f_i\colon \del H \to \Sigma$
such that $f_i^{-1}(\alpha_i)$ is a cut-system for $H$ for all $i$. Let $V_i$ be a tubular neighborhood of $p_i$ in $\del \Delta$ (recall that the $p_i$ are the vertices of~$\Delta$) and glue $H \times V_i$ along $\Sigma \times \Delta$ by the map $f_i\times \id$. Equip each $V_i$ with a standard simplex decomposition, with barycenter $p_i$ and a vertex $p_{ij}$ on the edge $[p_i,p_j]$ of $\Delta$. For each $i<j$, we have a $2$--sected $3$--manifold given by $M=(H\times \{p_{ij}\})\cup (\Sigma \times [p_{ij},p_{ji}]) \cup (H\times \{p_{ji}\})$ which has $(\Sigma;\alpha^i,\alpha^j)$ as a Heegaard diagram. According to the definition, this is a diagram for some Heegaard splitting of the boundary of a $4$--dimensional $1$--handlebody $Z$, and by the induction hypothesis, we conclude that
the $2$--sected manifolds $\del Z$ and $M$ are diffeomorphic. Therefore we can glue $Z$ along $M$. We continue
the process: the constraint on subcollections of curves from Definition~\ref{def:diagram} and the uniqueness result
in dimension $n$ (induction hypothesis) ensure that we always get connected sums of $S^1\times S^k$
along which we can glue $1$--handlebodies.

For the uniqueness part, assume two PL (resp. smooth) $(n+1)$--manifolds $W$ and $W'$ share a common multisection diagram.
It means $W$ and $W'$ can be multisected with central surfaces $\Sigma$ and $\Sigma'$ and defining collections of curves $(\alpha_j^i)_{1\leq j\leq g}$ and $(\beta_j^i)_{1\leq j\leq g}$, for $1\leq i\leq n$, such that there is a PL--homeomorphism (resp. diffeomorphism) $h:\Sigma\to\Sigma'$ satisfying $h(\alpha_j^i)=\beta_j^i$. First extend $h$ to collar neighborhoods of the boundaries of the $3$--dimensional handlebodies of the multisection, as a product with the identity in the normal direction to the boundary. Then the identification of the diagram curves implies that $h$ extends to the $3$--spine. At that point, $h$ is defined on the boundary of the $4$--dimensional pieces. Extend it again to collar neighborhoods of these boundaries, and then use Montesinos (resp. Laudenbach--Po\'enaru's) result (see above) to extend it to the $4$--spine. The process can be continued by induction using Cavicchioli--Hegenbarth's result (Theorem \ref{thm:cav:heg}), ending at the $6$--spine in the smooth case.
\end{proof}

\section{Examples}
\label{sec:examples}

\subsection{Genus 0}
Example~\ref{ex:sphere} gives a genus--$0$  multisection of $S^{n+1}$. By Theorem~\ref{thm:diagramVSmultisection}, $S^{n+1}$ is the only PL--manifold admitting a genus--$0$ $n$--section.

\subsection{Genus 1}
Consider the projection $\R^{k+1}\to\R^{k-1}$ forgetting the last two coordinates. The image of the unit sphere $S^k$ is the unit ball $B^{k-1}$. Viewed as a simplex, this $B^{k-1}$ admits the standard decomposition of Definition~\ref{def:simplex}, whose preimage is the {\em standard decomposition} of $S^k$. This decomposition can be written $S^k=\cup_{i=1}^kB_i$ where $\cap_{i\in I}B_i$ is a $(k-|I|+1)$--disk for any non-empty $I\subsetneq\{1,\dots,k\}$ and $\cap_{i=1}^kB_i\cong S^1$.

\begin{lemma} \label{lemma:spheregenusone}
 For $n\geq2$, the sphere $S^{n+1}$ admits at least $\lfloor\frac n2\rfloor$ non-homeomorphic genus--$1$ $n$--sections.
\end{lemma}
\begin{proof}
 For $0<k<n$, we have
 $$S^{n+1}=\partial(B^{k+1}\times B^{n-k+1})=(S^k\times B^{n-k+1})\cup(B^{k+1}\times S^{n-k}).$$
 We use the standard decompositions of $S^k$ and $S^{n-k}$. This provides a decomposition of $S^{n+1}$ into $n$ pieces that are either balls or genus--$1$ handlebodies. It is easily checked that this defines a multisection. Moreover, all the pieces of the multisection are genus--$0$ or genus--$1$ handlebodies and, for a given $k\leq\frac n2$, the highest dimension of a genus--$1$ piece is $n-k+2$. Note that replacing $k$ with $n-k$ simply exchanges the factors in the product.
\end{proof}

\begin{lemma} \label{lemma:S1timesSn}
 The manifold $S^1\times S^n$ admits a genus--$1$ $n$--section.
\end{lemma}
\begin{proof}
 We use the standard decomposition of $S^n$ and consider its product with $S^1$.
\end{proof}

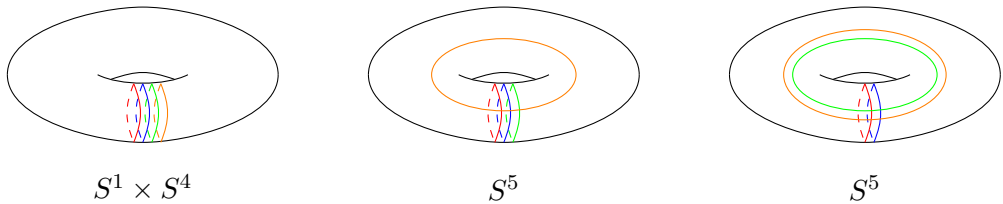
\begin{figure}[htb]
\begin{center}
\begin{tikzpicture} [scale=0.6]
\begin{scope}
 \draw (0,0) .. controls +(0,1) and +(-1,0) .. (3,1.5) .. controls +(1,0) and +(0,1) .. (6,0);
 \draw (0,0) .. controls +(0,-1) and +(-1,0) .. (3,-1.5) .. controls +(1,0) and +(0,-1) .. (6,0);
 \draw (2,0) ..controls +(0.5,-0.25) and +(-0.5,-0.25) .. (4,0);
 \draw (2.3,-0.1) ..controls +(0.6,0.2) and +(-0.6,0.2) .. (3.7,-0.1);
 \draw (3,-2.5) node {$S^1\times S^4$};
 \draw[orange] (3.4,-0.2) ..controls +(0.2,-0.5) and +(0.2,0.5) .. (3.4,-1.5);
 \draw[dashed,orange] (3.4,-0.2) ..controls +(-0.2,-0.5) and +(-0.2,0.5) .. (3.4,-1.5);
 \draw[green] (3.2,-0.2) ..controls +(0.2,-0.5) and +(0.2,0.5) .. (3.2,-1.5);
 \draw[dashed,green] (3.2,-0.2) ..controls +(-0.2,-0.5) and +(-0.2,0.5) .. (3.2,-1.5);
 \draw[blue] (3,-0.2) ..controls +(0.2,-0.5) and +(0.2,0.5) .. (3,-1.5);
 \draw[dashed,blue] (3,-0.2) ..controls +(-0.2,-0.5) and +(-0.2,0.5) .. (3,-1.5);
 \draw[red] (2.8,-0.2) ..controls +(0.2,-0.5) and +(0.2,0.5) .. (2.8,-1.5);
 \draw[dashed,red] (2.8,-0.2) ..controls +(-0.2,-0.5) and +(-0.2,0.5) .. (2.8,-1.5);
\end{scope}
\begin{scope} [xshift=8cm]
 \draw (0,0) .. controls +(0,1) and +(-1,0) .. (3,1.5) .. controls +(1,0) and +(0,1) .. (6,0);
 \draw (0,0) .. controls +(0,-1) and +(-1,0) .. (3,-1.5) .. controls +(1,0) and +(0,-1) .. (6,0);
 \draw (2,0) ..controls +(0.5,-0.25) and +(-0.5,-0.25) .. (4,0);
 \draw (2.3,-0.1) ..controls +(0.6,0.2) and +(-0.6,0.2) .. (3.7,-0.1);
 \draw (3,-2.5) node {$S^5$};
 \draw[green] (3.2,-0.2) ..controls +(0.2,-0.5) and +(0.2,0.5) .. (3.2,-1.5);
 \draw[dashed,green] (3.2,-0.2) ..controls +(-0.2,-0.5) and +(-0.2,0.5) .. (3.2,-1.5);
 \draw[blue] (3,-0.2) ..controls +(0.2,-0.5) and +(0.2,0.5) .. (3,-1.5);
 \draw[dashed,blue] (3,-0.2) ..controls +(-0.2,-0.5) and +(-0.2,0.5) .. (3,-1.5);
 \draw[red] (2.8,-0.2) ..controls +(0.2,-0.5) and +(0.2,0.5) .. (2.8,-1.5);
 \draw[dashed,red] (2.8,-0.2) ..controls +(-0.2,-0.5) and +(-0.2,0.5) .. (2.8,-1.5);
 \draw[orange] (3,0)ellipse(1.6 and 0.8);
\end{scope}
\begin{scope} [xshift=16cm]
 \draw (0,0) .. controls +(0,1) and +(-1,0) .. (3,1.5) .. controls +(1,0) and +(0,1) .. (6,0);
 \draw (0,0) .. controls +(0,-1) and +(-1,0) .. (3,-1.5) .. controls +(1,0) and +(0,-1) .. (6,0);
 \draw (2,0) ..controls +(0.5,-0.25) and +(-0.5,-0.25) .. (4,0);
 \draw (2.3,-0.1) ..controls +(0.6,0.2) and +(-0.6,0.2) .. (3.7,-0.1);
 \draw (3,-2.5) node {$S^5$};
 \draw[blue] (3.2,-0.2) ..controls +(0.2,-0.5) and +(0.2,0.5) .. (3.2,-1.5);
 \draw[dashed,blue] (3.2,-0.2) ..controls +(-0.2,-0.5) and +(-0.2,0.5) .. (3.2,-1.5);
 \draw[red] (3,-0.2) ..controls +(0.2,-0.5) and +(0.2,0.5) .. (3,-1.5);
 \draw[dashed,red] (3,-0.2) ..controls +(-0.2,-0.5) and +(-0.2,0.5) .. (3,-1.5);
 \draw[green] (3,0)ellipse(1.6 and 0.8);
 \draw[orange] (3,0)ellipse(1.8 and 1);
\end{scope}
\end{tikzpicture}
\caption{Genus--$1$ quadrisection diagrams} \label{fig:gen1quad}
\end{center}
\end{figure}

The multisections of Lemma~\ref{lemma:spheregenusone} have diagrams with two groups of $k$ and $n-k$ parallel curves respectively, where two curves from distinct groups meet at exactly one point and transversely; the multisection given for $S^1\times S^n$ has a diagram with $n$ parallel curves, see Figure~\ref{fig:gen1quad}.

\begin{lemma}
 For $n>3$, the multisections of Lemmas~\ref{lemma:spheregenusone} and~\ref{lemma:S1timesSn} are the only genus--$1$ $n$--sections, up to PL--homeomorphism.
\end{lemma}
\begin{proof}
 For a genus--$1$ multisection diagram, the subdiagrams defined by three curves have to be diagrams for $S^4$ or $S^1\times S^3$. It follows that the curves are necessarily divided into at most two groups of parallel curves. This concludes thanks to Theorem~\ref{thm:diagramVSmultisection}.
\end{proof}

\begin{corollary}
 For $n>3$, the only PL--manifolds of dimension $n+1$ admitting genus--$1$ multisections are $S^{n+1}$ and $S^1\times S^n$.
\end{corollary}

Recall that there are more genus--$1$ manifolds for smaller $n$: all lens spaces in dimension $3$, $\CP^2$ and $\overline\CP^2$ in dimension $4$.

\subsection{Some genus--$n$ multisections}

Figure~\ref{fig:genusn} shows an $n$--section diagram of genus $n$ such that the associated manifolds have a fundamental group of order $p$ for a given $p>0$. It implies that there are infinitely many $(n+1)$--dimensional PL--manifolds with a multisection of genus $n$. These diagrams are a generalization of the trisection diagrams of spun lens spaces obtained by Meier in \cite{Meier}. Rudy Dissler recently proved that these diagrams can be realized more generally by higher-dimensional spun lens spaces, namely manifolds $(L^\circ\times S^k)\cup(S^2\times B^{k+1})$, where $L^\circ$ is a lens space with an open $3$--ball removed \cite{Dissler}.

\begin{figure}[htb]
\begin{center}
\begin{tikzpicture} [scale=0.5]
\begin{scope}
 \draw (0,0) circle (6.2);
 \foreach \t/\c in {0/red,90/blue,180/green,270/orange} {
 \draw[rotate=\t,yshift=-3cm] \trou;
 \draw[rotate=\t,\c] (-3.1,0) ellipse (0.8 and 1.3);
 \draw[rotate=\t,\c] (0.6,-3) .. controls +(1,0.5) and +(-0.5,-1) .. (3,-0.6);
 \draw[rotate=\t,\c,dashed] (0.6,-3) .. controls +(0.5,1) and +(-1,-0.5) .. (3,-0.6);
 \draw[rotate=\t+90,\c] (0.4,-2.94) .. controls +(1,0.5) and +(-0.5,-1) .. (2.94,-0.4);
 \draw[rotate=\t+90,\c,dashed] (0.4,-2.94) .. controls +(0.5,1) and +(-1,-0.5) .. (2.94,-0.4);
 \draw[rotate=\t,\c,dashed] (3.2,-0.6) .. controls +(1,-0.4) and +(-1,0.2) .. (6.05,-1.4);
 \draw[rotate=\t,\c,dashed] (3.25,-0.4) .. controls +(1,-0.4) and +(-1,0.2) .. (6.1,-1.2);
 \draw[rotate=\t,\c,dashed] (3.3,-0.2) .. controls +(1,-0.4) and +(-1,0.2) .. (6.1,-1);
 \draw[rotate=\t,\c,dashed] (3.3,0) .. controls +(1,-0.4) and +(-1,0.2) .. (6.15,-0.8);
 \draw[rotate=\t,\c,dashed] (3.3,0.2) .. controls +(1,-0.4) and +(-1,0.2) .. (6.15,-0.6);
 \draw[rotate=\t,\c,dashed] (3.25,0.4) .. controls +(1,-0.4) and +(-1,0.2) .. (6.18,-0.4);
 \draw[rotate=\t,\c] (6.05,-1.4) .. controls +(-1,0.6) and +(1,-0.1) .. (3.25,-0.4);
 \draw[rotate=\t,\c] (6.1,-1.2) .. controls +(-1,0.6) and +(1,-0.1) .. (3.3,-0.2);
 \draw[rotate=\t,\c] (6.1,-1) .. controls +(-1,0.6) and +(1,-0.1) .. (3.3,0);
 \draw[rotate=\t,\c] (6.15,-0.8) .. controls +(-1,0.6) and +(1,-0.1) .. (3.3,0.2);
 \draw[rotate=\t,\c] (6.15,-0.6) .. controls +(-1,0.6) and +(1,-0.1) .. (3.25,0.4);
 \draw[rotate=\t,\c] (6.18,-0.4) .. controls +(-0.4,3) and +(4,0) .. (0,5.6) .. controls +(-3,0) and +(0,3) .. (-5,0) .. controls +(0,-3) and +(-3,0) .. (0,-4.4) .. controls +(4,0) and +(1,-0.1) .. (3.2,-0.6);
 }
 \draw (0,-7.5) node {$n=4$ and $p=6$};
\end{scope}
\begin{scope} [xshift=13.5cm]
 \draw (0,0) circle (6.2);
 \foreach \t/\c in {0/red,72/blue,144/green,216/orange,288/purple} {
 \draw[rotate=\t+18,yshift=-3cm] \trou;
 \draw[rotate=\t+36,\c] (-3.1,0) ellipse (0.8 and 1.3);
 \draw[rotate=\t,\c] (1.5,-2.67) .. controls +(0.6,0.5) and +(-0.4,-1) .. (3,-0.6);
 \draw[rotate=\t,\c,dashed] (1.5,-2.67) .. controls +(0.5,1) and +(-0.6,-0.5) .. (3,-0.6);
 \draw[rotate=\t+72,\c] (1.3,-2.68) .. controls +(0.6,0.5) and +(-0.4,-1) .. (2.94,-0.4);
 \draw[rotate=\t+72,\c,dashed] (1.3,-2.68) .. controls +(0.5,1) and +(-0.6,-0.5) .. (2.94,-0.4);
 \draw[rotate=\t+144,\c] (1.1,-2.68) .. controls +(0.6,0.5) and +(-0.4,-1) .. (2.9,-0.2);
 \draw[rotate=\t+144,\c,dashed] (1.1,-2.68) .. controls +(0.5,1) and +(-0.6,-0.5) .. (2.9,-0.2);
 \draw[rotate=\t,\c,dashed] (3.2,-0.6) .. controls +(1,-0.4) and +(-1,0.2) .. (6.1,-1.2);
 \draw[rotate=\t,\c,dashed] (3.3,-0.2) .. controls +(1,-0.4) and +(-1,0.2) .. (6.15,-0.8);
 \draw[rotate=\t,\c,dashed] (3.3,0.2) .. controls +(1,-0.4) and +(-1,0.2) .. (6.18,-0.4);
 \draw[rotate=\t,\c] (6.1,-1.2) .. controls +(-1,0.6) and +(1,-0.1) .. (3.3,-0.2);
 \draw[rotate=\t,\c] (6.15,-0.8) .. controls +(-1,0.6) and +(1,-0.1) .. (3.3,0.2);
 \draw[rotate=\t,\c] (6.18,-0.4) .. controls +(-0.4,3) and +(4,0) .. (0,5.6) .. controls +(-2.5,0) and +(0,3.5) .. (-5,0) .. controls +(0,-2.5) and +(-3,0) .. (0,-4.4) .. controls +(3,0) and +(1,-0.1) .. (3.2,-0.6);
 }
 \draw (0,-7.5) node {$n=5$ and $p=3$};
\end{scope}
\end{tikzpicture}
\end{center}
\caption{Diagrams of $n$--sections with fundamental group of order $p$}
\label{fig:genusn}
\end{figure}
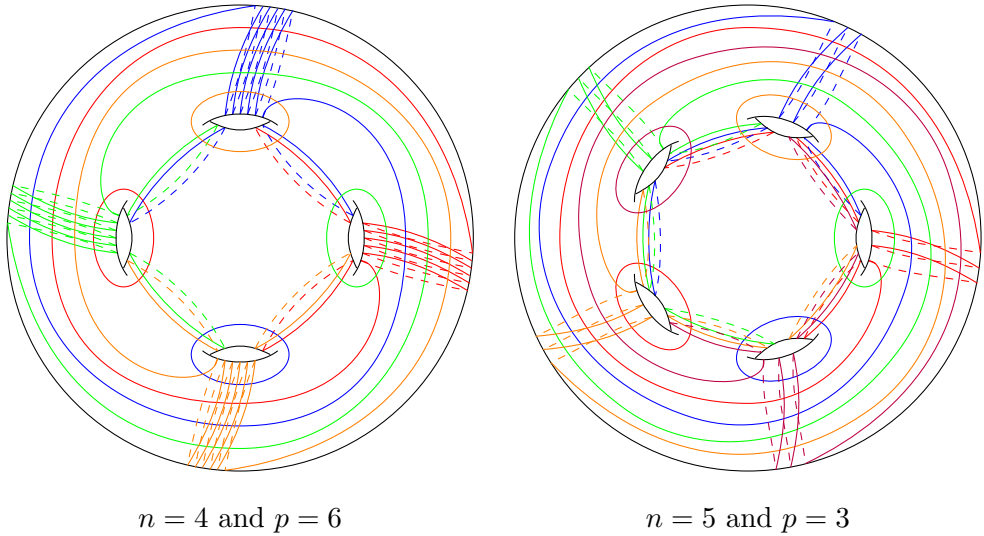

\subsection{The manifold $S^2\times S^3$}
To get a quadrisection of $S^2\times S^3$, we first decompose $S^3$ as a union $\cup_{i=1}^4B_i$ of $3$--balls $B_i$, where each $B_i\cap B_j$ is a $2$--disk, each $B_i\cap B_j\cap B_k$ is an interval and the global intersection is made of $2$ points, see the left image of Figure~\ref{fig:S2xS3}. Now we fix four disjoint $2$--disks~$D_i$, $1\leq i \leq 4$, on $S^2$ and we set $W_i=\big((S^2\setminus\mathrm{Int}(D_i))\times B_i\big)\cup(D_{i+1}\times B_{i+1})$. This defines a genus--$3$ quadrisection of $S^2\times S^3$ where $4$--dimensional pieces have genus $1$ and $5$--dimensional pieces have genus $0$, see the diagram on the right image of Figure~\ref{fig:S2xS3}. More generally, surface bundles are shown to admit multisections in \cite{Mmulti}.

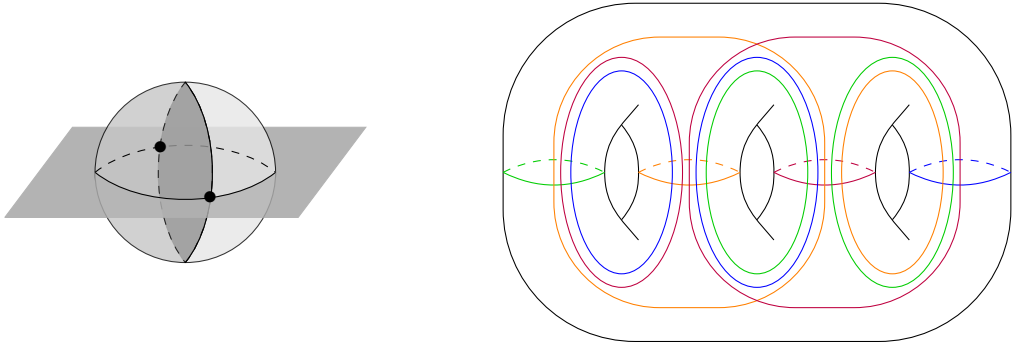
\begin{figure}[htb]
\begin{center}
\begin{tikzpicture}
\begin{scope}[scale=0.6]
 \draw[gray!60,fill=gray!60] (-4,-1) -- (2.5,-1) -- (4,1) -- (-2.5,1) -- (-4,-1);
 \draw[fill=gray!20,opacity=0.8] (0,-2) arc (-90:90:2) .. controls +(-0.8,-1) and +(-0.8,1) .. (0,-2);
 \draw[fill=gray!45,opacity=0.8] (0,2) arc (90:270:2) .. controls +(0.8,1) and +(0.8,-1) .. (0,2);
 \draw[dashed] (-2,0) .. controls +(1,0.8) and +(-1,0.8) .. (2,0);
 \draw[dashed,fill=gray!80,opacity=0.8] (0,-2) .. controls +(0.8,1) and +(0.8,-1) .. (0,2) .. controls +(-0.8,-1) and +(-0.8,1) .. (0,-2);
 \draw (0,-2) .. controls +(0.8,1) and +(0.8,-1) .. (0,2);
 \draw[gray!60,fill=gray!60,opacity=0.8] (-2,0) .. controls +(1,-0.8) and +(-1,-0.8) .. (2,0) -- (2,-1) -- (-2,-1) -- (-2,0);
 \draw (-2,0) .. controls +(1,-0.8) and +(-1,-0.8) .. (2,0);
 \foreach \s in {-1,1} {
 \draw[scale=\s] (0.55,-0.55) node {$\bullet$};}
\end{scope}
\begin{scope} [xshift=4cm,scale=0.45]
 \draw[rounded corners=50pt] (0.5,0) -- (0.5,-5) -- (15.5,-5) -- (15.5,5) -- (0.5,5) -- (0.5,0);
 \foreach \x in {4,8,12} {
 \draw (\x+0.5,2) .. controls +(-0.5,-0.6) and +(0,1) .. (\x-0.5,0) .. controls +(0,-1) and +(-0.5,0.6) .. (\x+0.5,-2);
 \draw (\x,1.4) .. controls +(0.5,-0.6) and +(0,0.3) .. (\x+0.5,0) .. controls +(0,-0.3) and +(0.5,0.6) .. (\x,-1.4);}
 \foreach \x/\c in {0/vert,4/orange,8/purple,12/blue} {
 \draw[color=\c] (\x+0.5,0) .. controls +(1,-0.5) and +(-1,-0.5) .. (\x+3.5,0);
 \draw[color=\c,dashed] (\x+0.5,0) .. controls +(1,0.5) and +(-1,0.5) .. (\x+3.5,0);}
 \foreach \x/\c in {4/blue,8/vert,12/orange} {
 \draw[color=\c] (\x,0) ellipse (1.5 and 3);}
 \foreach \x/\c in {4/purple,8/blue,12/vert} {
 \draw[color=\c] (\x,0) ellipse (1.8 and 3.4);}
 \foreach \x/\c in {0/orange,4/purple} {
 \draw[rounded corners=40pt,color=\c] (\x+2,0) -- (\x+2,-4) -- (\x+10,-4) -- (\x+10,4) -- (\x+2,4) -- (\x+2,0);}
\end{scope}
\end{tikzpicture}
\caption{Decomposition of $S^3$ and quadrisection diagram of $S^2\times S^3$} \label{fig:S2xS3}
\end{center}
\end{figure}

\section{Stabilizations}
\label{sec:stabilizations}

The stabilization move along an arc that occurs for Heegaard splittings and trisections naturally generalizes to higher-dimensional multisections.
In dimension $5$ and higher, we need to introduce higher-order stabilization moves, not only along arcs, but also along higher-dimensional disks.
The operation will be performed in an arbitrarily small neighborhood of a point of the central surface, so we may focus
on the local model $\R^2\times \R^{n-1}$ with the multisection induced by the standard decomposition of $\R^{n-1}$ (see Definition \ref{def:simplex} and Example \ref{ex:sphere}).

\begin{definition}
\begin{itemize}
 \item[]
 \item A \textit{half-disk} is a disk $\Delta$ whose boundary $\del \Delta$ is piecewise smooth and decomposed in two disks $\del_-\Delta$ and $\del_+\Delta$ which intersect precisely along their boundary. If $\Delta$ is $1$--dimensional, then $\del_-\Delta=\emptyset$.
 \item A half-disk $\Delta$ embedded in a manifold with boundary $W$ is \emph{standard} if $\Delta\cap\del W=\del_- \Delta$ and $\Delta$ is transverse to $\del W$. Note that if $\del W=\emptyset$, then $\Delta$ is necessarily $1$--dimensional.
 \item A half-disk $\Delta_I \subset W_I$ in a multisected manifold is \emph{$\MM$--standard} if $\Delta_J=\Delta_I\cap W_J$ is standard in $W_J$ for all $J\supset I$.
 \item A disk $D_I \subset W_I$ in a multisected manifold is \emph{boundary-parallel (with respect to the multisection)} if there is an $\MM$--standard half-disk $\Delta_I\subset W_I$ with $\del_+\Delta_I=D_I$ (see Figure~\ref{fig:halfdisk}).
\end{itemize}
\end{definition}

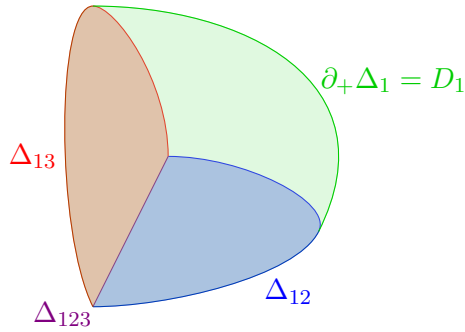
\begin{figure}[htb]
\begin{center}
\begin{tikzpicture}
 \draw[blue,fill=blue!30] (0,0) .. controls +(1,0) and +(-0.3,-0.6) .. (3,1) .. controls +(0.2,0.4) and +(1,0) .. (1,2) (2.6,0.2) node {$\Delta_{12}$};
 \draw[red,fill=red!30] (0,0) .. controls +(-0.5,1) and +(-0.5,0) .. (0,4) .. controls +(0.3,0) and +(0,1) .. (1,2) (-0.8,2) node {$\Delta_{13}$};
 \draw[color=violet] (0,0) -- (1,2) (-0.4,-0.1) node {$\Delta_{123}$};
 \draw[vert,fill=vert!40,opacity=0.3] (0,0) .. controls +(-0.5,1) and +(-0.5,0) .. (0,4) .. controls +(2,0) and +(1,2) .. (3,1) .. controls +(-0.3,-0.6) and +(1,0) .. (0,0);
 \draw[vert] (0,4) .. controls +(2,0) and +(1,2) .. (3,1) (4,3) node {$\partial_+\Delta_1=D_1$};
\end{tikzpicture}
\caption{For $n=3$, a half-disk $\Delta_1$ and its boundary} \label{fig:halfdisk}
\end{center}
\end{figure}

\begin{lemma} \label{lemma:halfdisks}
Let $W=\cup_i W_i$ be a closed multisected $(n+1)$--manifold. Let $I$ be a non-empty subset of $\{1,\dots,n\}$.
Any two $\MM$--standard half-disks $\Delta_I$, $\Delta'_I$ are isotopic among $\MM$--standard half-disks.
In particular, any two boundary-parallel disks are isotopic.
\end{lemma}
\begin{proof}
All disks are isotopic in a connected manifold and standard half-disks in a manifold with boundary
with the same negative boundary are isotopic relative to this negative boundary.

We use this fact and an induction procedure to prove the result.
Starting with $\Delta_{1,\dots,n}$ and $\Delta'_{1,\dots,n}$ which are embedded arcs in the connected surface $W_{1,\dots,n}$, and hence isotopic, we extend the corresponding isotopy of $W_{1,\dots,n}$ to an ambient isotopy of $W$
preserving the multisection.
Hence we can assume $\Delta_{1,\dots,n}=\Delta'_{1,\dots,n}$.

We continue with $J$ such that $I\subset J$ and $|J|=n-1$. We have $\del_-\Delta_J=\del_-\Delta'_J$
and hence $\Delta_J$ and $\Delta'_J$ are isotopic relative to the negative boundary. We may thus find
an ambient isotopy preserving the multisection, which is the identity on $W_{1,\dots,n}$
and takes $\Delta_J$ to $\Delta'_J$. The process goes on inductively.
\end{proof}

\begin{definition}
A tubular neighborhood $N$ of a codimension--$2$ boundary-parallel disk $D\subset W_I$ is in \emph{good position} if
it can be written $N=D\times E$ where $E$ is a disk of dimension $|I|+1$ together with a decomposition $E=\cup_{i\in I}E_i$ such that:
\begin{itemize}
 \item for all non-empty $J\subset I$, $\cap_{i\in J}E_i$ is a disk with corners of dimension $|I|-|J|+2$,
 \item for all non-empty $J\subset\{1,\dots,n\}$, we have $N\cap W_J=(D\cap W_J) \times E_{I\cap J}$, where $E_\emptyset=E$.
\end{itemize}
\end{definition}

\begin{lemma}\label{lem:tubnghd}
If two tubular neighborhoods $N$ and $N'$ of a codimension--$2$ boundary-parallel disk $D_I\subset W_I$ in a multisected manifold
are in good position, then they are isotopic relative to $D_I$ among such disks.
\end{lemma}
\begin{proof}
We proceed step by step as in Lemma~\ref{lemma:halfdisks}, using the uniqueness of tubular neighborhoods. We start with $(D\cap\Sigma)\times E_I$ and $(D\cap\Sigma)\times E'_I$, which are isotopic relative to $D\cap\Sigma$. We extend to an ambient isotopy of $W$. We continue with $(D\cap\Sigma)\times E_J$ for $J\subset I$ such that $|J|=|I|-1$, which is isotopic to $(D\cap\Sigma)\times E'_J$ relative to their negative boundary. The process goes on until the isotopy identifies $(D\cap\Sigma)\times E$ to $(D\cap\Sigma)\times E'$. Then we proceed similarly to isotope $(D\cap W_J)\times E$ onto $(D\cap W_J)\times E'$ for $J$ such that $|J|=n-1$, and so on.
\end{proof}

\begin{definition}
Let $I$ be a non-empty proper subset of $\{1,\dots,n\}$, $D\subset W_I$ a codimension--$2$ boundary-parallel disk together with a tubular neighborhood $N=D\times E$ in good position. For $j\notin I$, define an \emph{$I$--stabilization} of the multisection by removing $D\times E_i$ from $W_i$ for all $i\in I$ and adding $D\times E$ to $W_j$, then smoothing the corners.
\end{definition}

\begin{lemma}
 For each $j$, the $I$--stabilization of a multisection is again a multisection which is well-defined up to isotopy. Moreover, up to isotopy, it is independent of the choice of $j$.
\end{lemma}
\begin{proof}
In view of Lemmas~\ref{lemma:halfdisks} and~\ref{lem:tubnghd}, the choice of a boundary-parallel disk and its tubular neighborhood in good position is unique up to isotopy. So the only point to check is that different choices of $W_i$ to which we add $D\times E$ result in isotopic multisections. In fact, more generally we can extend the decomposition $\del D=\cup_{j\not \in I} (D\cap W_j)$ to similar decompositions $D=\cup_{j\not \in I} D_j$ into (possibly empty) disks $D_j$ and add $D_j\times E$ to each $W_j$. Using $1$--parameter families of such decompositions we show that the resulting multisections are all isotopic.
\end{proof}

\begin{proposition}
 The $I$--stabilization and the $I^c$--stabilization of a given multisection are isotopic multisections. Furthermore, the genus of $W_J$ is increased by $1$ if $J\supset I$ or $J\supset I^c$, and otherwise remains the same.
\end{proposition}
\begin{proof}
 Assume the $I$--stabilization adds a neighborhood $D\times E$ of a boundary-parallel disk $D$ to $W_{i_0}$ with $i_0\notin I$. Let $\Delta\subset W_I$ be a standard half-disk for $D$. The neighborhood $D\times E$ can be extended to a neighborhood $N(\Delta)$ which is a standard half-disk in $\cup_{i\in I}W_i$. Performing the $I$--stabilization amounts to first adding $N(\Delta)$ to $W_{i_0}$ and then performing an $I^c$--stabilization.

 If $I\subset J$, then $W_J\subset W_I$ and a neighborhood of a boundary-parallel codimension--$2$ disk is carved off $W_J$, which is equivalent to adding a $1$--handle. The stabilization can be equivalently performed on $W_{I^c}$, so that the effect is the same if $I^c\subset J$. In the remaining cases, the move acts by digging along the boundary, a process that can be realized progressively by an isotopy, so the genus of $W_J$ is unchanged.
\end{proof}

\begin{corollary}\label{cor:inducedstab}
 The $I$--stabilization has the following effect on the induced multisection of $\partial W_J$:
 \begin{itemize}
  \item a connected sum with the genus--$1$ multisection of $S^1\times S^{n-|J|}$ if $J\supset I$ or $J\supset I^c$,
  \item an $(I\cup J)$--stabilization otherwise.
 \end{itemize}
\end{corollary}

\begin{lemma}
 The $I$--stabilization of a multisection $\mathcal M$ can be obtained by a suitable connected sum with a genus--$1$ multisection of $S^{n+1}$.
\end{lemma}
\begin{proof}
 The $I$--stabilization involves a codimension--$2$ boundary-parallel disk $D$ in $W_I$; it is performed in a regular neighborhood $B$ of a half-disk $\Delta$ such that $\partial_+\Delta=D$. Cutting along the boundary of $B$ realizes $\mathcal M$ as a connected sum of $\mathcal M$ with the standard genus--$0$ multisection of $S^{n+1}$. After stabilization, it becomes a connected sum of $\mathcal M$ with a genus--$1$ multisection of $S^{n+1}$.
\end{proof}

Looking at the proof of Lemma~\ref{lemma:spheregenusone}, we see that the genus--$1$ multisection of $S^{n+1}$ that appears in the lemma above is the one with $k=|I|$ in Lemma~\ref{lemma:spheregenusone}.

\section{Four-manifold quadrisections}
\label{sec:4Dquadri}

In \cite{IN}, Islambouli and Naylor introduced a generalization of trisections for $4$--manifolds: they define a multisection of a $4$--manifold $X$ as a decomposition $X=\cup_{1\leq i\leq n}X_i$ where
\begin{itemize}
 \item each $X_i$ is a $4$--dimensional handlebody,
 \item $\Sigma=\cap_{1\leq i\leq n}X_i$ is a closed surface,
 \item $H_i=X_i\cap X_{i+1}$ is a $3$--dimensional handlebody for each $i$, where $X_{n+1}=X_1$,
 \item $X_i\cap X_j=\Sigma$ otherwise.
\end{itemize}
With such a multisection is associated a {\em multisection diagram} defined by the central surface $\Sigma$ with a cut-system for each $3$--dimensional handlebody $H_i$. The diagram determines the associated manifold up to diffeomorphism. Note that, in contrast with $n$--section diagrams of $(n+1)$--manifolds, the cyclic order of the cut-systems in a multisection diagram of a $4$--manifold is relevant.

\begin{example} \label{ex:quad4D}
 In Figure~\ref{fig:gen1quad}, the leftmost diagram represents $S^1\times S^3$ and the middle diagram represents $S^4$. The rightmost diagram represents $S^4$ if the red and blue curves correspond to consecutive $H_i$'s and $S^2\times S^2$ otherwise. \\
 The genus--$1$ multisection diagrams of $S^4$ are the diagrams with two non-empty groups of parallel curves, where two curves from distinct groups meet at exactly one point and transversely, and the curves of a given group correspond to consecutive $H_i$'s.
\end{example}

Here, we are specifically interested in $4$--manifold quadrisections, as they appear in $5$--manifold quadrisections. Actually, a quadrisection of a $5$--manifold $W$ induces a quadrisection of the $4$--dimensional ``middle level'' $\partial(W_1\cup W_2)$, with the same quadrisection diagram.

In \cite{IN}, stabilisation moves are defined for $4$--manifold multisections, which amount to connect summing with the genus--$1$ multisections of $S^4$. They also introduce the so-called UPW move, which modifies the number of pieces in the multisection. They prove that any two multisections of the same $4$--manifold are related by a sequence of stabilization moves and UPW moves. Here, we also consider a {\em fake stabilization} move defined as the connected sum with the genus--$1$ multisected $S^2\times S^2$ represented by the middle diagram in Figure \ref{figQuadrisections}.

\begin{figure}[htb]
\begin{center}
\begin{tikzpicture}
\begin{scope} [scale=0.6]
 \draw (0,0) circle (2);
 \draw[purple] (-2,0) -- (0,0);
 \draw[blue] (0,0) -- (2,0);
 \draw[vert] (0,-2) -- (0,0);
 \draw[orange] (0,0) -- (0,2);
 \draw (1.8,-1.8) node {$\scriptstyle{W_{23}}$} (1.8,1.8) node {$\scriptstyle{W_{13}}$} (-1.8,-1.8) node {$\scriptstyle{W_{24}}$} (-1.8,1.8) node {$\scriptstyle{W_{14}}$};
\end{scope}
\begin{scope} [xshift=2.2cm,scale=0.8]
 \draw (0,0) .. controls +(0,1) and +(-1,0) .. (3,1.5) .. controls +(1,0) and +(0,1) .. (6,0);
 \draw (0,0) .. controls +(0,-1) and +(-1,0) .. (3,-1.5) .. controls +(1,0) and +(0,-1) .. (6,0);
 \draw (2,0) ..controls +(0.5,-0.25) and +(-0.5,-0.25) .. (4,0);
 \draw (2.3,-0.1) ..controls +(0.6,0.2) and +(-0.6,0.2) .. (3.7,-0.1);
 \draw[blue] (3.2,-0.2) ..controls +(0.2,-0.5) and +(0.2,0.5) .. (3.2,-1.5);
 \draw[dashed,blue] (3.2,-0.2) ..controls +(-0.2,-0.5) and +(-0.2,0.5) .. (3.2,-1.5);
 \draw[purple] (3,-0.2) ..controls +(0.2,-0.5) and +(0.2,0.5) .. (3,-1.5);
 \draw[dashed,purple] (3,-0.2) ..controls +(-0.2,-0.5) and +(-0.2,0.5) .. (3,-1.5);
 \draw[vert] (3,0)ellipse(1.6 and 0.8);
 \draw[orange] (3,0)ellipse(1.8 and 1);
\end{scope}
\begin{scope} [xshift=7.5cm,scale=0.45]
 \draw[rounded corners=50pt] (0.5,0) -- (0.5,-5) -- (15.5,-5) -- (15.5,5) -- (0.5,5) -- (0.5,0);
 \foreach \x in {4,8,12} {
 \draw (\x+0.5,2) .. controls +(-0.5,-0.6) and +(0,1) .. (\x-0.5,0) .. controls +(0,-1) and +(-0.5,0.6) .. (\x+0.5,-2);
 \draw (\x,1.4) .. controls +(0.5,-0.6) and +(0,0.3) .. (\x+0.5,0) .. controls +(0,-0.3) and +(0.5,0.6) .. (\x,-1.4);}
 \foreach \x/\c in {0/vert,4/orange,8/purple,12/blue} {
 \draw[color=\c] (\x+0.5,0) .. controls +(1,-0.5) and +(-1,-0.5) .. (\x+3.5,0);
 \draw[color=\c,dashed] (\x+0.5,0) .. controls +(1,0.5) and +(-1,0.5) .. (\x+3.5,0);}
 \foreach \x/\c in {4/blue,8/vert,12/orange} {
 \draw[color=\c] (\x,0) ellipse (1.5 and 3);}
 \foreach \x/\c in {4/purple,8/blue,12/vert} {
 \draw[color=\c] (\x,0) ellipse (1.8 and 3.4);}
 \foreach \x/\c in {0/orange,4/purple} {
 \draw[rounded corners=40pt,color=\c] (\x+2,0) -- (\x+2,-4) -- (\x+10,-4) -- (\x+10,4) -- (\x+2,4) -- (\x+2,0);}
\end{scope}
\end{tikzpicture}
\caption{Quadrisection diagrams for $S^2\times S^2$} \label{figQuadrisections}
\end{center}
\end{figure}
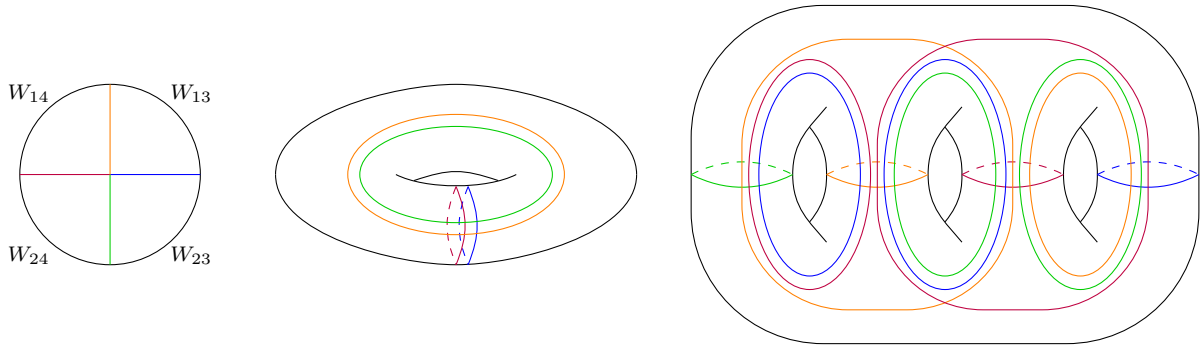

\begin{proposition}
 The effect of an $I$--stabilization of a quadrisected $5$--manifold $W$ on the induced quadrisection of $\partial(W_1\cup W_2)$ is a fake stabilization when $I=\{12\},\{34\}$ and a stabilization otherwise.
\end{proposition}
\begin{proof}
 A $5$--dimensional stabilization can be performed by a connected sum with a genus--$1$ quadrisection of $S^5$, whose diagram is given in Figure~\ref{fig:gen1quad}.
\end{proof}

The next result shows that two multisections of the same $4$--manifold with the same number of pieces are not always related by stabilizations only, hence it can be indeed necessary to modify the number of pieces. This remains true when allowing fake stabilizations.

\begin{proposition}
 The manifold $S^2\times S^2$ admits two distinct quadrisections which are not related by a sequence of (fake) stabilizations.
\end{proposition}

\begin{proof}
 Figure \ref{figQuadrisections} represents two quadrisection diagrams of the $4$--manifold $S^2\times S^2$. Viewed as diagrams of $5$--manifold quadrisections, the middle diagram represents $S^5$ and the rightmost one represents $S^2\times S^3$. Since (fake) stabilization moves on the $4$--dimensional middle level are induced by $5$--dimensional stabilizations, and since $S^5$ is not diffeomorphic to $S^2\times S^3$, these two quadrisections of $S^2\times S^2$ are not related by (fake) stabilizations.
\end{proof}

\section{Existence in dimension five}
\label{sec:5Dexistence}

The existence of trisections for $4$--dimensional smooth manifolds was proved by Gay and Kirby in \cite{GayKirby}. The existence proof we give here in dimension $5$ is closer in spirit to the proof of Lambert-Cole and Miller given in \cite[Section~2.1]{LCMiller}.

To study the existence of $5$--dimensional quadrisections, we shall mainly work from a $4$--dimensional quadrisection on the ``middle level'' of a Morse function. The $4$--dimensional quadrisections that appear in this setting have specific properties.

\begin{definition}
 Let $X$ be a closed connected orientable $4$--manifold with two transversely embedded (possibly disconnected) closed surfaces $L_2$ and~$L_3$. A {\em special quadrisection of $X$ relative to $L_2$ and $L_3$} is a quadrisection $X=X_{23}\cup X_{24}\cup X_{14} \cup X_{13}$ (in this cyclic order) such that $X_{23}\cup X_{24}$ is a tubular neighborhood of $L_2$ together with $1$--handles attached, and similarly for $L_3$ in $X_{13}\cup X_{23}$.
\end{definition}

\begin{lemma}\label{lem:existquadri}
Let $X$ be a closed connected orientable $4$--manifold and $L_2$ and $L_3$ two closed surfaces (possibly disconnected)
in $X$ which are transverse to each other. Then there exists a special quadrisection $X=W_{23}\cup W_{24}\cup W_{14}\cup W_{13}$ relative to $L_2$ and $L_3$.
\end{lemma}
\begin{proof}
We pick a regular neighborhood $A$ of $L_2\cup L_3$ and a Morse function $f\colon X\setminus \Int A \to \R$ (constant on $\del A$ and attaining its minimum
there).
Pick a collection of $2$--disks properly embedded in $A$ and which intersect $L_2$ or $L_3$ transversely
in a single point. If we have chosen enough such disks, removing a tubular neighborhood of each one from $A$ results in a handlebody $A'$
inside of which $L'_2=L_2\cap A'$ and $L'_3=L_3\cap A'$ are properly embedded. We would then get $A$ back by attaching $2$--handles along the link of circles
$\del L'_2\cup \del L'_3$. Next consider in $\del A$ the descending spheres of critical points of $f$ of index $0$ and $1$; we may assume by general position
that they lie in $\del A'\setminus (\del L'_2\cup \del L'_3)$. We obtain $A''$ by attaching these $0$-- and $1$--handles to $A'$. Note that $L'_2$ and $L'_3$
are still properly embedded in $A''$. Then we look at descending spheres of critical points of $f$ of index $2$. Again by general position
we may assume they lie in $\del A'' \setminus (\del L'_2\cup \del L'_3)$. We decompose each of these $2$--handles as a $1$--handle and two $2$--handles (each of which would cancel the $1$--handle). And finally we define $W_{23}$ to be $A''$ together with these $1$--handles attached (one for each critical point of $f$ of index $2$).
In the boundary of $W_{23}$ we have the link $\del L'_2\cup \del L'_3$ and the attaching spheres of the $2$--handles which now come in pairs
which we denote $M_2$ and~$M_3$. We choose a Heegaard splitting $\del W_{23}=W_{123}\cup W_{234}$ such that $W_{234}$ (resp. $W_{123}$) is a tubular neighborhood
of $\del L'_2\cup M_2$ (resp. $\del L'_3\cup M_3$) with $1$--handles attached (this is done by choosing
an ordered Morse function on the cobordism going from the boundary of one of these tubular neighborhoods to the other, and choosing a level
set separating critical points of index $1$ and $2$).

Now we define $W_{24}$ as $W_{234}$ with $2$--handles attached along $\del L'_2\cup M_2$,
with a corner along $\del W_{234}$ and a new boundary piece $W_{124}=\del W_{24}\setminus W_{234}$.
We claim that $W_{24}$ and $W_{124}$ are handlebodies. First, $W_{24}$ has $1$ and $2$--handles but all $2$--handles are cancelled against $1$--handles. Second, $W_{124}$ is the result of surgery on $W_{234}$, it is again a tubular neighborhood of a link of circles, with $1$--handles attached, so a handlebody.
Observe that $W_{23}\cup W_{24}$ is a tubular neighborhood of $L_2$ with $1$--handles attached: the extra $2$--handles attached along $M_2$
are cancelled against $1$--handles of $W_{23}$ by construction. We define $W_{13}$ similarly by attaching $2$--handles along $\partial L_3'\cup M_3$ instead.

The remaining piece $W_{14}=W\setminus \Int(W_{13}\cup W_{23}\cup W_{24})$ is also a handlebody since it has a handle
decomposition with only $0$--handles and $1$--handles (corresponding to the critical points of $f$ of index $3$ and $4$).
\end{proof}

\begin{theorem}\label{thm:existquadri}
Every closed connected orientable $5$--manifold $W$ admits a quadrisection.
\end{theorem}

\begin{proof}
We choose an ordered Morse function $f\colon W \to \R$ together with a Morse--Smale gradient vector field.
Consider a level set $X=f^{-1}(c)$ separating the critical points of index $2$ and $3$. In $X$, we have two links
of $2$--spheres $L_2^+$ and $L_3^-$ which correspond respectively to ascending spheres of all critical points of index $2$ and descending spheres of all critical points of index $3$ (see Figure~\ref{fig:quad}). The Morse--Smale condition guarantees that $L_2^+$ and $L_3^-$ are transverse to each other.
The main step is to construct a special quadrisection $X=W_{23}\cup W_{24}\cup W_{14}\cup W_{13}$ relative to $L_2^+$ and~$L_3^-$.
Observe that $X$ is connected and orientable as $W$ is, so the existence of this quadrisection is guaranteed by Lemma~\ref{lem:existquadri} applied to $L_2^+$ and $L_3^-$.

\begin{figure}[htb]
\begin{center}
\begin{tikzpicture}
\begin{scope}
 \draw (0,0) circle (2);
 \draw (-2,0) -- (2,0) (0,-2) -- (0,2);
 \draw[blue] (1,0) ellipse (0.5 and 1);
 \draw[blue] (0.4,1) node {$L_3^-$};
 \draw[white] (1,-1) node {$\bullet$};
 \draw[red] (0,-1) ellipse (1 and 0.5);
 \draw[red] (-1,-0.4) node {$L_2^+$};
 \draw (1.8,-1.8) node {$W_{23}$} (1.8,1.8) node {$W_{13}$} (-1.8,-1.8) node {$W_{24}$} (-1.8,1.8) node {$W_{14}$};
\end{scope}
\begin{scope} [yscale=0.5,xshift=7cm]
 \draw[fill,red] (2,-4) -- (2,0) arc (40:124.5:2) -- (-0.664,-1.63) arc (124.5:140:2) -- (-1.065,-4) arc (-140:-40:2);
 \draw[fill,blue!80,opacity=0.8] (-0.664,0.37) arc (124.5:140:2) -- (-1.065,-2) arc (-140:-55.5:2) -- (1.6,-0.37);
 \draw[fill,orange,opacity=0.7] (-1.065,2) arc (-140:-68:2) -- (1.216,0.57) -- (-0.28,-0.57) arc (-112:-140:2);
 \draw[fill,vert,opacity=0.7] (2,2) arc (40:68:2) -- (-0.28,1.43) -- (-0.28,-0.57) arc (-112:-40:2);
 \draw[fill,orange,opacity=0.6] (2,2) -- (2,4)  arc (40:140:2) -- (-1.065,2) arc (-140:-40:2);
 \foreach \y/\d in {-4/dashed,-2/dashed,0/dashed,2/dashed,4/} {
 \draw (2,\y) arc (-40:-140:2);
 \draw[\d] (2,\y) arc (40:140:2);}
 \foreach \x in {-1.065,2} \draw (\x,-4) -- (\x,4);
 \foreach \y in {2,0} {
 \draw[dashed] (1.6,-0.37-\y) -- (-0.664,0.37-\y);
 \draw[dashed] (-0.28,-0.57+\y) -- (1.216,0.57+\y);}
 \draw (1.6,-0.37) -- (1.6,-2.37) (-0.28,-0.57) -- (-0.28,1.43);
 \draw[dashed] (-0.664,0.37) -- (-0.664,-1.63) (1.216,0.57) -- (1.216,2.57);
 \draw[->] (4,-4) -- (4,4);
 \draw (3,0.5) node {$\longrightarrow$} node[above] {$f$};
 \foreach \i/\y in {1/-2,2/0,3/2}
 \draw[dashed] (2.2,\y) -- (4,\y) node[right] {$c_\i$};
 \draw (4.8,0) node {$=c$};
 \foreach \i/\y/\c in {1/-3/red,2/-1/blue,3/1/vert,4/3/orange}
 \draw[\c] (-1.5,\y) node {$W_\i$};
 \draw (-0.56,-0.05) node {$\scriptstyle{W_{24}}$};
 \draw (0.6,-0.4) node {$\scriptstyle{W_{23}}$};
 \draw (1.5,0) node {$\scriptstyle{W_{13}}$};
 \draw (0.4,0.4) node {$\scriptstyle{W_{14}}$};
\end{scope}
\end{tikzpicture}
\caption{Quadrisecting a $5$--manifold\\
{\footnotesize{The left hand side represents the middle level $f^{-1}(c)$. The right hand side represents the Morse function and the associated quadrisection: the level $c_i$ separates the critical points of index $i$ and $i+1$.}}} \label{fig:quad}
\end{center}
\end{figure}
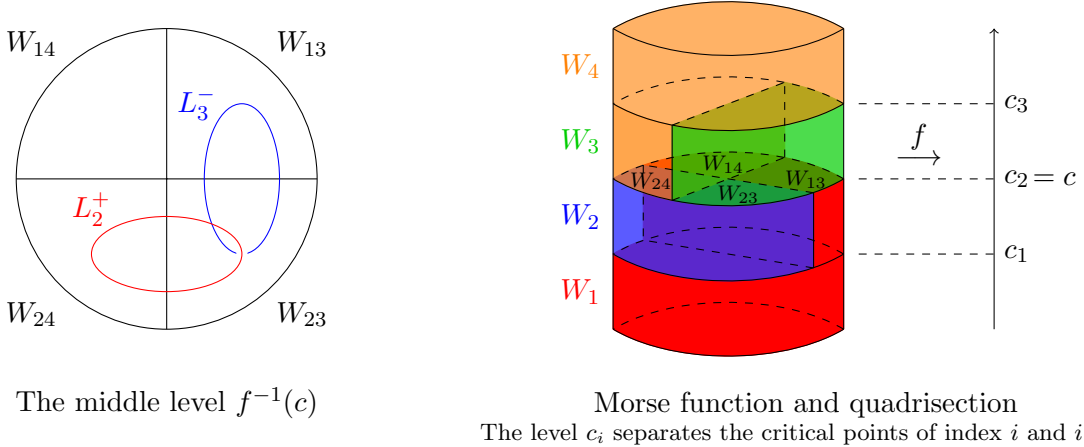

We then construct the quadrisection of $W$ as follows.
Attach $3$--handles downwards on $W_{23}\cup W_{24}$ along $L_2^+$ so as to obtain $W_2$, with a corner along $\del (W_{23}\cup W_{24})$.
We claim that $W_2$ is a handlebody: indeed $(W_{23}\cup W_{24})\times [0,1]$ has handles of index $0$, $1$ and $2$ but all handles of index $2$ are cancelled
by the attachment along $L_2^+$. Moreover the new boundary component $W_{12}=\del W_2\setminus \Int(W_{23}\cup W_{24}))$
is also a handlebody: each tubular neighborhood of a $2$--sphere of the link $L_2^+$ becomes a tubular neighborhood of a circle in $W_{12}$ (after surgery).
We define $W_1=\overline{\{f\leq c\}\setminus W_2}$. It is a handlebody since it is constructed from $0$--handles and $1$--handles (seen from below).

Now the situation is completely symmetric: we define $W_3$ and $W_4$ similarly in the region above $X$, attaching $3$--handles along $L_3^-$ instead of $L_2^+$.
\end{proof}

\begin{remark}
 As was pointed out by Maggie Miller, this existence result can be recovered using the trisections of cobordisms between $4$--manifolds developed in \cite{LCMiller}. The rough idea is to start with an ordered Morse function on the $5$--manifold $W$ (with one critical point of index $5$), define $W_1$ as a sublevel set containing the critical points of index $0$ and $1$, and then trisect the cobordism between $\partial W_1$ and the boundary of a small ball around the critical point of index $5$. A careful analysis of Lambert-Cole and Miller's proof shows that the construction can be handled in order to satisfy the requirements of a quadrisection.
\end{remark}

\section{More examples}
\label{sec:moreExs}

We shall illustrate the existence proof of the previous section with some examples; we use the notations from there.

\begin{example}
We start with $S^5$. Let $f:S^5\to[0,5]$ be a self-indexing Morse function with four critical points, of indices $0,1,2$ and $5$. The middle level set $X=f^{-1}\left(\frac52\right)$ is diffeomorphic to $S^4$ and $L_2^+$ is a trivially embedded $2$--sphere in $X$. We now quadrisect $X$ following the proof of Lemma~\ref{lem:existquadri}. We consider a handle decomposition of $X$ with a cancelling $2$/$3$--pair, where $L_2^+$ is the belt sphere of the $2$--handle. Hence $W_{23}$ is a $4$--ball meeting $L_2^+$ transversely along a disk. We have a Kirby diagram of $X$ with a single red circle $K$ associated to $L_2^+$: the diagram represents $\partial W_{23}$ and $K$ is its intersection with $L_2^+$. A Heegaard surface for $\partial W_{23}$ is given by the boundary of a tubular neighborhood of $K$, see Figure~\ref{figquadS5}. We immediately get a Heegaard diagram of $\partial W_{23}=W_{234}\cup W_{123}$. Noting that $W_{134}$ is isotopic to $W_{123}$ and that $W_{124}$ is obtained from $W_{234}$ by gluing a $2$--handle along $K$ (so that the red curve bounds a disk in $W_{124}$), we get the quadrisection diagram in Figure~\ref{figquadS5}.
\end{example}

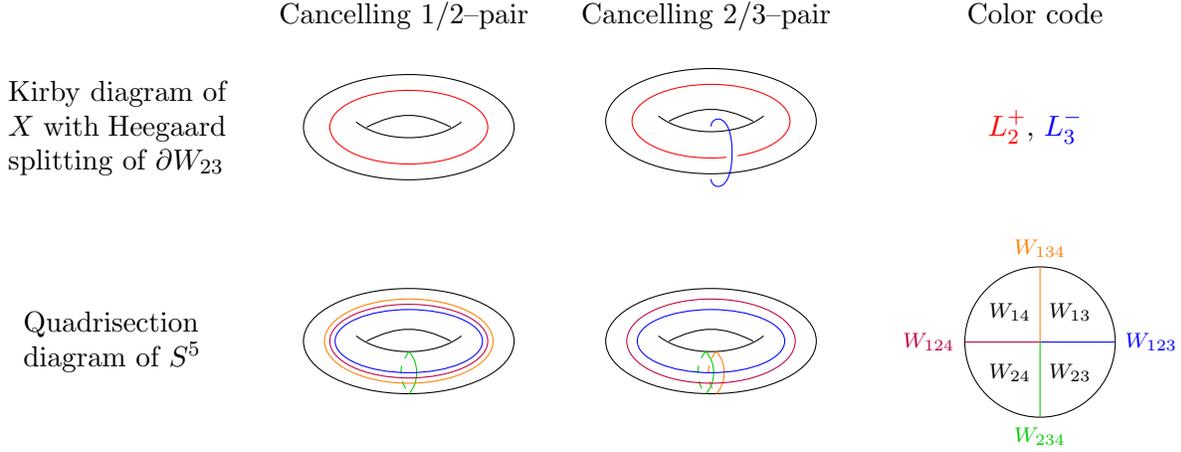
\begin{figure}[htb]
\begin{center}
\begin{tabular}{ccccccc}
 && Cancelling $1/2$--pair && Cancelling $2/3$--pair && Color code \\ &&&&&& \\
 \begin{minipage}{2.9cm}
 Kirby diagram of $X$ with Heegaard splitting of $\partial W_{23}$
 \end{minipage}
 &&
 \raisebox{-0.6cm}{
 \begin{tikzpicture} [scale=0.7]
  \draw (0,0) ellipse (2 and 1);
  \draw \buco;
  \draw[red] (0,0) ellipse (1.5 and 0.7);
 \end{tikzpicture}}
 &&
 \raisebox{-0.8cm}{
 \begin{tikzpicture} [scale=0.7]
  \draw (0,0) ellipse (2 and 1);
  \draw \buco;
  \draw[red] (0,0) ellipse (1.5 and 0.7);
  \draw[white] (0.4,-0.7) node {$\bullet$};
  \draw[blue] (0,-0.1) .. controls +(0,0.2) and +(0,0.8) .. (0.4,-0.6) .. controls +(0,-0.8) and +(0,-0.2) .. (0,-1.1);
 \end{tikzpicture}}
 && \color{red} $L_2^+$\color{black}, \color{blue} $L_3^-$\color{black} \\ &&&&&& \\
 \begin{minipage}{2.5cm}
 Quadrisection \\ diagram of $S^5$ \end{minipage}
 &&
 \raisebox{-0.6cm}{
 \begin{tikzpicture} [scale=0.7]
  \draw (0,0) ellipse (2 and 1);
  \draw \buco;
  \draw[vert] (0,-0.2) ..controls +(0.2,-0.2) and +(0.2,0.2) .. (0,-1);
  \draw[dashed,vert] (0,-0.2) ..controls +(-0.2,-0.3) and +(-0.2,0.3) .. (0,-1);
  \draw[purple] (0,0) ellipse (1.5 and 0.7);
  \draw[orange] (0,0) ellipse (1.6 and 0.8);
  \draw[blue] (0,0) ellipse (1.4 and 0.6);
 \end{tikzpicture}}
 &&
 \raisebox{-0.6cm}{
 \begin{tikzpicture} [scale=0.7]
  \draw (0,0) ellipse (2 and 1);
  \draw \buco;
  \draw[dashed,orange] (0.1,-0.2) ..controls +(-0.2,-0.2) and +(-0.2,0.2) .. (0.1,-1);
  \draw[dashed,vert] (-0.1,-0.2) ..controls +(-0.2,-0.2) and +(-0.2,0.2) .. (-0.1,-1);
  \draw[vert] (-0.1,-0.2) ..controls +(0.2,-0.2) and +(0.2,0.2) .. (-0.1,-1);
  \draw[orange] (0.1,-0.2) ..controls +(0.2,-0.2) and +(0.2,0.2) .. (0.1,-1);
  \draw[purple] (0,0) ellipse (1.6 and 0.8);
  \draw[blue] (0,0) ellipse (1.4 and 0.6);
 \end{tikzpicture}}
 &&
 \raisebox{-1.4cm}{
 \begin{tikzpicture} [scale=0.5]
  \draw (0,0) circle (2);
  \draw[purple] (-2,0) node[left] {$\scriptstyle{W_{124}}$} -- (0,0);
  \draw[blue] (0,0) -- (2,0) node[right] {$\scriptstyle{W_{123}}$};
  \draw[vert] (0,-2) node[below] {$\scriptstyle{W_{234}}$} -- (0,0);
  \draw[orange] (0,0) -- (0,2) node[above] {$\scriptstyle{W_{134}}$};
  \draw (0.8,-0.8) node {$\scriptstyle{W_{23}}$} (0.8,0.8) node {$\scriptstyle{W_{13}}$} (-0.8,-0.8) node {$\scriptstyle{W_{24}}$} (-0.8,0.8) node {$\scriptstyle{W_{14}}$};
 \end{tikzpicture}}
\end{tabular}
\caption{Quadrisecting $S^5$ with a Morse function} \label{figquadS5}
\end{center}
\end{figure}

\begin{example}
Again with $S^5$, we now consider a Morse function with four critical points, of indices $0,2,3$ and $5$. The middle level set $X$ is diffeomorphic to $S^2\times S^2$ with $L_2^+$ identified with $S^2\times\{*\}$ and $L_3^-$ with $\{*\}\times S^2$. We can then define $W_{23}\subset X$ as $D^2\times D^2$, where $*\in D^2$, and we have a Kirby diagram of $X$ given by the two circles associated to $L_2^+$ and $L_3^-$ in $\partial W_{23}$. We deduce the quadrisection diagram as previously, see Figure~\ref{figquadS5}.
\end{example}

\begin{example}
We turn to $S^2\times S^3$. It admits a Morse function with four critical points, of indices $0,2,3$ and~$5$. The middle level set $X$ is again diffeomorphic to $S^2\times S^2$, with this time $L_2^+$ and $L_3^-$ identified with $S^2\times\{p\}$ and $S^2\times\{q\}$ for two distinct points $p$ and $q$. Define $W_{23}$ as $D\times R$, where $D$ is a $2$--disk embedded in $S^2$ and $R$ is an embedded annulus containing $p$ and $q$. As such, $W_{23}$ is diffeomorphic to $S^1\times B^3$. Figure~\ref{fig:S2xS3bis} represents the attachment link of four $2$--handles in~$\partial W_{23}$. The blue circle $\partial D\times \{p\}$ is associated to $L_2^+$, the red circle $\partial D\times \{q\}$ is associated to $L_3^-$ and the other two are given by $\{*\}\times\partial R$. Define a Heegaard surface for $\partial W_{23}$ as follows: on the Kirby diagram of Figure~\ref{fig:S2xS3bis}, draw a horizontal plane in the middle, which represents a torus, and two other tori around the red and blue circle; then join these tori by two tubes in the obvious way. From this deduce the Heegaard splitting of $\partial W_{23}$ given by the green and blue curves on Figure~\ref{fig:S2xS3bis} ---we keep the color code of Figure~\ref{figquadS5}. Then complete the diagram to get a quadrisection diagram of $S^2\times S^3$. Note that it is diffeomorphic to the diagram of Figure~\ref{fig:S2xS3}.
\end{example}

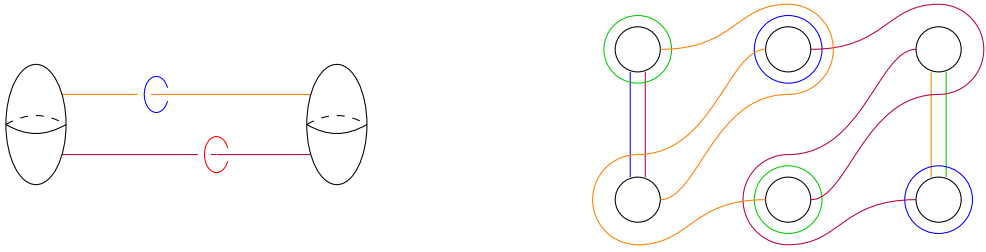
\begin{figure}[htb]
\begin{center}
\begin{tikzpicture}
\begin{scope} [scale=0.4]
 \draw[orange] (0,1) -- (4,1);
 \draw[purple] (0,-1) -- (6,-1);
 \drawwb{(4,1) ellipse (0.4 and 0.6);}
 \drawwr{(6,-1) ellipse (0.4 and 0.6);}
 \drawwp{(4,1) -- (10,1);}
 \drawwo{(6,-1) -- (10,-1);}
 \foreach \x in {0,10} {
 \begin{scope} [xshift=\x cm]
 \draw[fill=white] (0,0) ellipse (1 and 2);
 \draw (-1,0) .. controls +(0.7,-0.4) and +(-0.7,-0.4) .. (1,0);
 \draw[dashed] (-1,0) .. controls +(0.7,0.4) and +(-0.7,0.4) .. (1,0);
 \end{scope}}
\end{scope}
\begin{scope} [xshift=10cm]
 \foreach \x in {-2,0,2} \foreach \y in {-1,1}
 \draw (\x,\y) circle (0.3);
 \foreach \x/\c in {-2.1/blue,-1.9/purple,1.9/orange,2.1/vert}
 \draw[\c] (\x,-0.7) -- (\x,0.7);
 \foreach \x/\y/\c in {-2/1/vert,0/-1/vert,0/1/blue,2/-1/blue}
 \draw[\c] (\x,\y) circle (0.45);
 \foreach \t/\c in {0/purple,180/orange} {
 \draw[\c,rotate=\t] (0.3,1) .. controls +(1,0) and +(-0.7,0) .. (2,1.6) arc (90:-90:0.6) .. controls +(-1,0) and +(0.5,0) .. (0.3,-1);
 \draw[\c,rotate=\t] (1.7,-1) .. controls +(-1,0) and +(0.7,0) .. (0,-1.6) arc (270:90:0.6) .. controls +(1,0) and +(-0.5,0) .. (1.7,1);}
\end{scope}
\end{tikzpicture}
\caption{Kirby diagram of $S^2\times S^2$ and quadrisection diagram of $S^2\times S^3$}
\label{fig:S2xS3bis}
\end{center}
\end{figure}

\begin{example}
Our last example is the projective space $\RP^5$.
On $\RP^5$, we have a Morse function $f$ given by $f\big([x_0:\dots:x_5]\big)=\frac{\sum_{i=0}^5 ix_i^2}{\sum_{i=0}^5 x_i^2}$, with one critical point of each index.

\begin{lemma}
 The middle level set $X=f^{-1}\left(\frac52\right)$ is diffeomorphic to the quotient of $S^2\times S^2$ by the antipodal map (on both components simultaneously).
\end{lemma}
\begin{proof}
 We view $\RP^5$ as a quotient of $S^5$ and we work with homogeneous coordinates in $S^5$.
 In the CW--decomposition of $\RP^5$ associated to the Morse function $f$, the $2$--skeleton is $\{[x_0:x_1:x_2:0:0:0]\}$. The level set $X$ is isotopic to the boundary of a regular neighborhood of this $2$--skeleton; such a neighborhood is for instance $\{[x_0:\dots:x_5]\mid x_3^2+x_4^2+x_5^2\leq\frac12\}$. Thus $X\cong\{[x_0:\dots:x_5]\mid x_0^2+x_1^2+x_2^2=x_3^2+x_4^2+x_5^2\}$ is identified with the quotient of $S^2\times S^2\subset S^5$ by the antipodal map.
\end{proof}

In this description, the ascending sphere $L_2^+\subset X$ (resp. descending sphere $L_3^-\subset X$) of the critical point of index $2$ (resp. $3$) is the image of $\{(0,0,\pm1)\}\times S^2$ (resp. $S^2\times\{(0,0,\pm1)\}$).
They intersect exactly twice.

We will first determine a Kirby diagram of $X$ including the embedding of the spheres $L_2^+$ and~$L_3^-$, then deduce a quadrisection diagram of $X$ and $\RP^5$.

\begin{figure}[htb]
\begin{center}
\begin{tikzpicture} [xscale=0.6,yscale=0.5]
 \foreach \x in {0,8} {
 \begin{scope} [xshift=\x cm]
 \draw[blue,rounded corners=8pt] (2,3) -- (1,3) -- (1,6) -- (7,6) node[right] {$0$} -- (7,3) -- (3,3);
 \drawwr{(-1,4) -- (2,4) (3,4) -- (4,4) node[right] {$0$} -- (4,5) -- (-1,5) (9,4) -- (6,4) -- (6,5) -- (9,5);}
 \drawwb{(7,3.5) -- (7,4.5);}
 \draw[rounded corners=8pt] (-1,2) -- (2,2) (-1,0) -- (4,0) node[below] {$2$} -- (6,2) -- (9,2);
 \draww{(3,2) -- (4,2) -- (6,0) -- (9,0);}
 \draw (2,1.5) -- (3,1.5) -- (3,4.5) -- (2,4.5) -- (2,1.5);
 \draw (2.5,3) node {$+1$};
 \end{scope}}
 \foreach \x in {0,18} {
 \begin{scope} [xshift=\x cm]
 \draw[fill=white] (-1,3) ellipse (1 and 4);
 \draw (-2,3) .. controls +(0.7,-0.4) and +(-0.7,-0.4) .. (0,3);
 \draw[dashed] (-2,3) .. controls +(0.7,0.4) and +(-0.7,0.4) .. (0,3);
 \end{scope}}
 \draw[red] (13.7,5) node {$R$};
 \draw (15,1.5) node {$K$};
 \draw (15,-0.5) node {$J$};
\end{tikzpicture}
\caption{Kirby diagram of $S^2\times S^2$} \label{figKirbyS2xS2}
\end{center}
\end{figure}
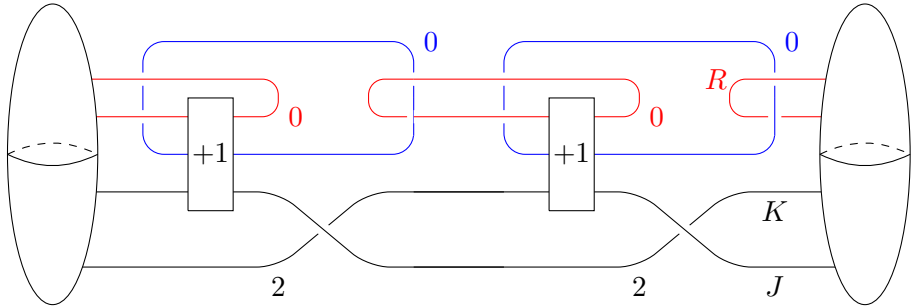

\begin{lemma}
 The Kirby diagram of Figure~\ref{figKirbyS2xS2} represents the closed manifold $S^2\times S^2$. Moreover, the $2$--spheres obtained from the red and blue components by gluing the core of the corresponding handle with an embedded $2$--disk in the $S^2\times S^1$ of the diagram can be identified with $\{*\}\times S^2$ and $S^2\times\{*\}$.
\end{lemma}
\begin{proof}
 Slide $R$ twice and $J$ once over $K$, so that the $2$--handle represented by $K$ cancels with the $1$--handle. After simplifications, this gives the left image of Figure~\ref{figKirbyS2xS2again}. Then slide a blue component over the other one; it gets split from the link. Do the same with red components. Finally slide the black component over the red and the blue successively. This gives the right image of Figure~\ref{figKirbyS2xS2again}.
\end{proof}

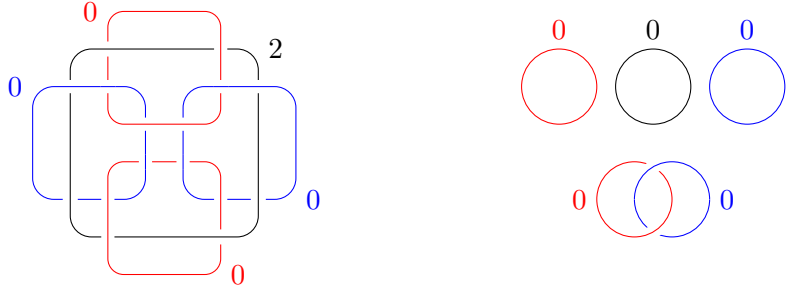
\begin{figure}[htb]
\begin{center}
\begin{tikzpicture} [scale=0.5]
 \draw[rounded corners=8pt] (1,3) -- (1,1) -- (6,1) -- (6,6) node[right] {$2$} -- (1,6) -- (1,3);
 \drawwb{(0.8,2) -- (0,2) -- (0,5) node[left] {$0$} -- (3,5) -- (3,2) -- (1.2,2);}
 \drawwb{(6.2,2) -- (7,2) node[right] {$0$} -- (7,5) -- (4,5) -- (4,2) -- (5.8,2);}
 \drawwr{(2,6.2) -- (2,7) node[left] {$0$} -- (5,7) -- (5,4) -- (2,4) -- (2,5.8);}
 \drawwr{(5,0.8) -- (5,0) node[right] {$0$} -- (2,0) -- (2,3) -- (5,3) -- (5,1.2);}
 \drawwb{(1.8,5) -- (2.2,5) (4.8,5) -- (5.2,5);}
 \drawwb{(3,2.8) -- (3,3.2) (4,2.8) -- (4,3.2);}
\begin{scope} [xshift=15cm,yshift=2cm]
 \draw[blue] (1,0) arc (-180:0:1);
 \drawwr{(1,0) circle (1);}
 \drawwb{(3,0) node[right] {$0$} arc (0:180:1);}
 \draw[red] (0,0) node[left] {$0$};
 \foreach \x/\c in {-1/red,1.5/black,4/blue}
 {\draw[\c] (\x,3) circle (1) (\x,4) node[above] {$0$};}
\end{scope}
\end{tikzpicture}
\caption{Other Kirby diagrams of $S^2\times S^2$} \label{figKirbyS2xS2again}
\end{center}
\end{figure}

\begin{figure}[htb]
\begin{center}
\begin{tikzpicture} [xscale=0.5,yscale=0.42]
\begin{scope}
 \draw[blue,rounded corners=8pt] (2,3) -- (1,3) -- (1,6) -- (7,6) node[right] {$0$} -- (7,3) -- (3,3);
 \drawwr{(-1,4) -- (2,4) (3,4) -- (4,4) node[right] {$0$} -- (4,5) -- (-1,5) (9,4) -- (6,4) -- (6,5) -- (9,5);}
 \drawwb{(7,3.5) -- (7,4.5);}
 \draw[rounded corners=8pt] (-1,2) -- (2,2) (-1,0) -- (4,0) node[below] {$3$} -- (6,2) -- (9,2);
 \draww{(3,2) -- (4,2) -- (6,0) -- (9,0);}
 \draw (2,1.5) -- (3,1.5) -- (3,4.5) -- (2,4.5) -- (2,1.5);
 \draw (2.5,3) node {$+1$};
 \foreach \x in {0,10} {
 \begin{scope} [xshift=\x cm]
 \draw[fill=white] (-1,3) ellipse (1 and 4);
 \draw (-2,3) .. controls +(0.7,-0.4) and +(-0.7,-0.4) .. (0,3);
 \draw[dashed] (-2,3) .. controls +(0.7,0.4) and +(-0.7,0.4) .. (0,3);
 \end{scope}}
\end{scope}
\begin{scope} [xshift=18cm,yshift=1cm]
 \draw[blue,rounded corners=8pt] (2,3) -- (1,3) -- (1,6) -- (7,6) node[right] {$0$} -- (7,3) -- (3,3);
 \drawwr{(-1,4) -- (2,4) (3,4) -- (4,4) node[right] {$0$} -- (4,5) -- (-1,5) (9,4) -- (6,4) -- (6,5) -- (9,5);}
 \drawwb{(7,3.5) -- (7,4.5);}
 \draw[rounded corners=8pt,purple] (-1,2) -- (2,2) (-1,0) .. controls +(4,0) and +(3,0) .. (0,-2.8) (10,-2.8) -- (6,-2.8) -- (4,-1.2) -- (4,0.2) node[left] {$3$} -- (6,2) -- (9,2);
 \drawwo{(3,2) -- (4,2) -- (6,0) -- (9,0);}
 \draw[orange] (0,-3.2) -- (10,-3.2);
 \draw (2,1.5) -- (3,1.5) -- (3,4.5) -- (2,4.5) -- (2,1.5);
 \draw (2.5,3) node {$+1$};
 \foreach \x in {0,10} {
 \begin{scope} [xshift=\x cm]
 \draw[fill=white] (-1,3) ellipse (1 and 4);
 \draw (-2,3) .. controls +(0.7,-0.4) and +(-0.7,-0.4) .. (0,3);
 \draw[dashed] (-2,3) .. controls +(0.7,0.4) and +(-0.7,0.4) .. (0,3);
 \end{scope}}
 \foreach \x in {0,10} {
 \begin{scope} [xshift=\x cm,yshift=-6cm]
 \draw[fill=white] (-1,3) ellipse (1 and 1.2);
 \draw (-2,3) .. controls +(0.7,-0.4) and +(-0.7,-0.4) .. (0,3);
 \draw[dashed] (-2,3) .. controls +(0.7,0.4) and +(-0.7,0.4) .. (0,3);
 \end{scope}}
\end{scope}
\end{tikzpicture}
\caption{Kirby diagrams of the $4$--manifold $X$} \label{figKirbyX}
\end{center}
\end{figure}
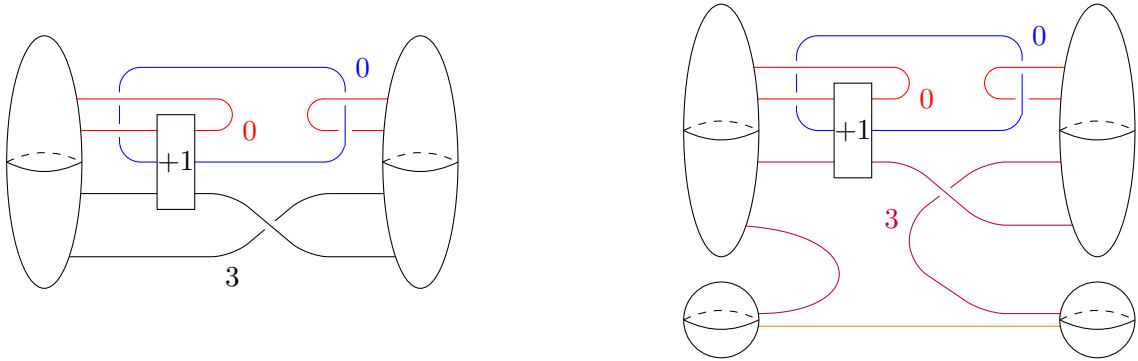

\begin{lemma}
 The Kirby diagrams of Figure~\ref{figKirbyX} represent the middle level set $X$.
\end{lemma}
\begin{proof}
 The $4$--manifold represented by the Kirby diagram on the left image of Figure~\ref{figKirbyX} has a double cover given by the Kirby diagram in Figure~\ref{figKirbyS2xS2}, thus this cover is $S^2\times S^2$ and the covering map identifies $\{*\}\times S^2$ with $\{-*\}\times S^2$ and $S^2\times\{*\}$ with $S^2\times\{-*\}$.
\end{proof}

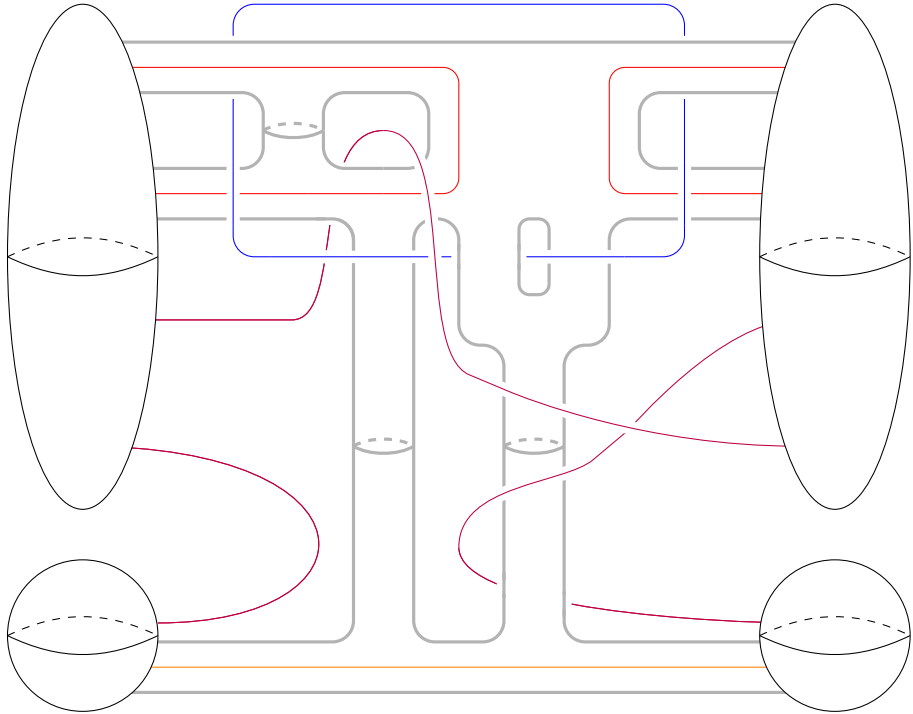
\begin{figure}[htb]
\begin{center}
\begin{tikzpicture} [yscale=0.84]
 \draw[blue,rounded corners=8pt] (3,3) --
 (1,3) -- (1,5.5) (1,6.5) -- (1,7) -- (7,7) -- (7,6.5) (7,5.5) -- (7,3) -- (3,3);
 \draw[gray!60,very thick,rounded corners=8pt] (-1,-3.9) -- (9,-3.9)
 (-1,-3.1) -- (2.6,-3.1) -- (2.6,2.9) (3.4,2.9) -- (3.4,-3.1) -- (4.6,-3.1) -- (4.6,1.6) -- (4,1.6) -- (4,3.6) -- (3.4,3.6) -- (3.4,3.1) (2.6,3.1) -- (2.6,3.6) -- (-1,3.6)
 (9,-3.1) -- (5.4,-3.1) -- (5.4,1.6) -- (6,1.6) -- (6,3.6) -- (9,3.6)
 (-1,6.4) -- (9,6.4)
 (-1,4.4) -- (1.4,4.4) -- (1.4,5.6) -- (-1,5.6)
 (3,4.4) -- (3.6,4.4) -- (3.6,5.6) -- (2.2,5.6) -- (2.2,5)
 (9,4.4) -- (6.4,4.4) -- (6.4,5.6) -- (9,5.6);
 \draw[gray!60,very thick,rounded corners=4pt] (4.8,3) -- (4.8,2.4) -- (5.2,2.4) -- (5.2,3.6) -- (4.8,3.6) -- (4.8,3);
 \drawwr{(-1,4) -- (4,4) -- (4,6) -- (-1,6) (9,4) -- (6,4) -- (6,6) -- (9,6);}
 \foreach \x in {1,7}
 \drawwb{(\x,3.5) -- (\x,4.5);}
 \drawwb{(5,3) -- (6.2,3);}
 \draw[rounded corners=8pt,purple] (-1,0) .. controls +(4,0) and +(3,0) .. (0,-2.8)
 (4,-1.6) .. controls +(0,-1) and +(-1,0) .. (8.5,-2.8)
 (-1,2) -- (1.5,2) .. controls +(1,0) and +(-1,0) .. (3,5) ;
 \draw[white,line width=5pt]
 (4.5,-2.2) -- (5.5,-2.5)
 (3.9,3) -- (4.9,3)
 (2.3,3.5) -- (2.45,4.5) ;
 \drawwb{(1.5,3) -- (2.5,3);}
 \drawwr{(2,4) -- (3,4);}
 \drawwo{(4,-1.6) .. controls +(0,1) and +(-0.6,-0.6) .. (6,0) .. controls +(0.3,0.3) and +(-1,0) .. (8.5,2);}
 \drawwo{(3,5) .. controls +(1,0) and +(-1,0.5) .. (4.4,1) .. controls +(0.5,-0.25) and +(-2,0) .. (8.5,0);}
 \draw[orange] (-1,-3.5) -- (9,-3.5);
 \draw[gray!60,very thick,rounded corners=8pt]
 (4.6,-2.4) -- (4.6,-2) (5.4,-2.3) -- (5.4,-2.7)
 (4,2.8) -- (4,3.2) (4.8,2.8) -- (4.8,3.2)
 (2.1,3.6) -- (2.6,3.6) -- (2.6,3.1)
 (2.2,5) -- (2.2,4.4) -- (3,4.4);
 \foreach \x/\y in {2.6/0,4.6/0,1.4/5} {
 \draw[gray!60,very thick]
 (\x,\y) .. controls +(0.2,-0.15) and +(-0.2,-0.15) .. (\x+0.8,\y);
 \draw[gray!60,very thick,dashed]
 (\x,\y) .. controls +(0.2,0.15) and +(-0.2,0.15) .. (\x+0.8,\y);
 }
 \foreach \x in {0,10} {
 \begin{scope} [xshift=\x cm]
 \draw[fill=white] (-1,3) ellipse (1 and 4);
 \draw (-2,3) .. controls +(0.7,-0.4) and +(-0.7,-0.4) .. (0,3);
 \draw[dashed] (-2,3) .. controls +(0.7,0.4) and +(-0.7,0.4) .. (0,3);
 \end{scope}}
 \foreach \x in {0,10} {
 \begin{scope} [xshift=\x cm,yshift=-6cm]
 \draw[fill=white] (-1,3) ellipse (1 and 1.2);
 \draw (-2,3) .. controls +(0.7,-0.4) and +(-0.7,-0.4) .. (0,3);
 \draw[dashed] (-2,3) .. controls +(0.7,0.4) and +(-0.7,0.4) .. (0,3);
 \end{scope}}
\end{tikzpicture}
\caption{Heegaard splitting of $\partial W_{23}$, with the attaching circles of the $2$--handles} \label{figHeegaardSurface}
\end{center}
\end{figure}

\begin{figure}[htb]
\begin{center}
\begin{tikzpicture}
 \draw[gray!60,very thick,rounded corners=8pt] (4,-2) -- (0,-2) -- (0,-1) -- (8,-1) -- (8,-2) -- (4,-2)
 (4,-3) -- (-1,-3) -- (-1,0) -- (2,0) -- (2,3.5) -- (-3,3.5) -- (-3,10.5) -- (4,10.5)
 (3,3) -- (3,0) -- (5,0) -- (5,1.5) -- (4,1.5) -- (4,3.5) -- (3,3.5) -- (3,3)
 (4,-3) -- (9,-3) -- (9,0) -- (6,0) -- (6,1.5) -- (7,1.5) -- (7,3.5) -- (11,3.5) -- (11,10.5) -- (4,10.5)
 (4,7.5) -- (0,7.5) -- (0,6.5) -- (8,6.5) -- (8,7.5) -- (4,7.5)
 (4,9.5) -- (-2,9.5) -- (-2,4.5) -- (1,4.5) -- (1,5.5) -- (-1,5.5) -- (-1,8.5) -- (4,8.5)
 (3.5,4.5) -- (4,4.5) -- (4,5.5) -- (2.5,5.5) -- (2.5,4.5) -- (3.5,4.5)
 (4,9.5) -- (10,9.5) -- (10,4.5) -- (6.5,4.5) -- (6.5,5.5) -- (9,5.5) -- (9,8.5) -- (4,8.5)
 (5,3) -- (5,2.5) -- (6,2.5) -- (6,3.5) -- (5,3.5) -- (5,3);
 \foreach \x/\y/\c in {2/1.6/green,2/1.8/purple,4/2.9/green,4/3.1/purple} {
 \draw[\c]
 (\x,\y) .. controls +(0.2,-0.15) and +(-0.2,-0.15) .. (\x+1,\y);
 \draw[\c,dashed]
 (\x,\y) .. controls +(0.2,0.15) and +(-0.2,0.15) .. (\x+1,\y);
 }
 \foreach \x/\y/\c in {1/4.9/green,1/5.1/purple} {
 \draw[\c]
 (\x,\y) .. controls +(0.2,-0.15) and +(-0.2,-0.15) .. (\x+1.5,\y);
 \draw[\c,dashed]
 (\x,\y) .. controls +(0.2,0.15) and +(-0.2,0.15) .. (\x+1.5,\y);
 }
 \foreach \x/\y/\c in {3.9/-3/green,4.1/-3/blue,4/7.5/green} {
 \draw[\c]
 (\x,\y) .. controls +(0.15,0.2) and +(0.15,-0.2) .. (\x,\y+1);
 \draw[\c,dashed]
 (\x,\y) .. controls +(-0.15,0.2) and +(-0.15,-0.2) .. (\x,\y+1);
 }
 \foreach \x/\y/\c in {5.4/3.5/green,5.6/3.5/purple,1.6/3.5/blue} {
 \draw[\c]
 (\x,\y) .. controls +(0.2,0.5) and +(0.2,-0.5) .. (\x,\y+3);
 \draw[\c,dashed]
 (\x,\y) .. controls +(-0.2,0.5) and +(-0.2,-0.5) .. (\x,\y+3);
 }
 \draw[purple,rounded corners=12pt]
 (4,-0.5) -- (8.5,-0.5) -- (8.5,-2.5) -- (-0.5,-2.5) -- (-0.5,-0.5) -- (4,-0.5)
 (4,10) -- (-2.5,10) -- (-2.5,4) -- (4.5,4) -- (4.5,6) -- (-0.5,6) -- (-0.5,8) -- (8.5,8) -- (8.5,6) -- (6,6) -- (6,4) -- (10.5,4) -- (10.5,10) -- (4,10);
 \draw[blue,rounded corners=12pt]
 (4,9.8) -- (-2.3,9.8) -- (-2.3,4.2) -- (1.3,4.2) -- (1.3,5.8) -- (-0.7,5.8) -- (-0.7,8.2) -- (8.7,8.2) -- (8.7,5.8) -- (6.2,5.8) -- (6.2,4.2) -- (10.3,4.2) -- (10.3,9.8) -- (4,9.8)
 (2.2,5) -- (2.2,4.2) -- (4.3,4.2) -- (4.3,5.8) -- (2.2,5.8) -- (2.2,5)
 (2.7,3) -- (2.7,-0.3) -- (5.3,-0.3) -- (5.3,1.8) -- (4.3,1.8) -- (4.3,3.8) -- (2.7,3.8) -- (2.7,3)
 (4.7,3) -- (4.7,2.2) -- (6.3,2.2) -- (6.3,3.8) -- (4.7,3.8) -- (4.7,3);
 \draw[orange,rounded corners=12pt]
 (6,0.5) -- (5.5,-0.3) -- (8.7,-0.3) -- (8.7,-2.7) -- (4,-2.7) -- (-0.7,-2.7) -- (-0.7,-0.3) -- (2.5,-0.3) -- (2.5,3.8) -- (-2.7,3.8) -- (-2.7,10.2) -- (10.7,10.2) -- (10.7,3.8) -- (6.5,3.8) -- (6.5,2) -- (4.5,2) -- (4.5,4) -- (3,4.5)
 (2,3.2) -- (2.3,3.7) -- (-2.8,3.7) -- (-2.8,10.3) -- (10.8,10.3) -- (10.8,3.7) -- (6.7,3.7) -- (6.7,1.7) -- (5.7,1.7) -- (5,1)
 (1,5.3) -- (2,5.3) -- (2,4) -- (3.7,3.5)
 (5.3,3.5) -- (5.1,4.5) -- (6.5,5) ;
 \draw[orange,dashed,rounded corners=12pt]
 (3,4.5) .. controls +(-0.1,-0.5) and +(0.4,0.3) .. (2,3.2)
 (5,1) .. controls +(0.5,-0.1) and +(-0.3,0.3) .. (6,0.5)
 (3.7,3.5) .. controls +(0.3,0.3) and +(-0.3,0.3) .. (5.3,3.5)
 (6.5,5) .. controls +(-1,1) and +(1,0) .. (3.5,6.2) .. controls +(-0.5,0) and +(1,0.8) .. (1,5.3) ;
\end{tikzpicture}
\caption{Quadrisection diagram of $\RP^5$\\{\footnotesize The missing orange curves can be deduced from the blue cut-system.}} \label{figDiagramRP5}
\end{center}
\end{figure}
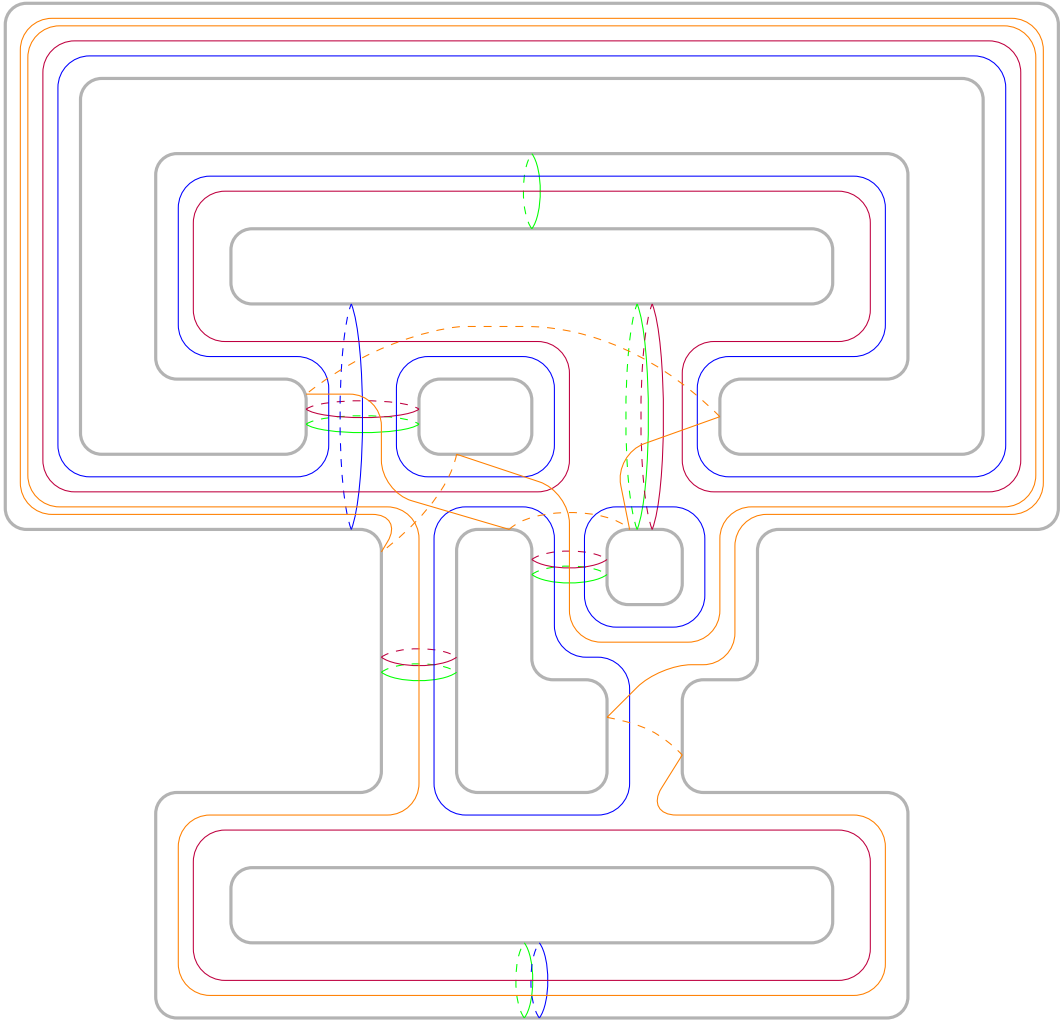

Following the proof of Theorem~\ref{thm:existquadri}, we want to produce a special quadrisection of $X$ relative to $L_2^+$ and $L_3^-$. A diagram of this $4$--dimensional quadrisection is also a quadrisection diagram of $\RP^5$. We will construct such a quadrisection $X=W_{23}\cup W_{24}\cup W_{14} \cup W_{13}$ following the proof of Lemma~\ref{lem:existquadri}, and using the notations of this proof. First, we split the $2$--handle represented by the black circle in the left image of  Figure~\ref{figKirbyX}, which gives the Kirby diagram in the right image of Figure~\ref{figKirbyX}. The latter is the boundary of $W_{23}$ with the attaching link defining $W_{24}$ given by the red circle (associated to $L_2^+$) and the orange one, and the attaching link defining $W_{13}$ made of the blue circle (associated to $L_3^-$) and the purple one. Figure~\ref{figHeegaardSurface} gives a suitable Heegaard splitting of this $\partial W_{23}$. We first draw on the Heegaard surface the diagram representing this splitting $\partial W_{23}=W_{234}\cup W_{123}$, see Figure~\ref{figDiagramRP5} where the cut systems for $W_{234}$ and $W_{123}$ are in green and blue respectively. Then a cut system for $W_{124}$, in purple, is obtained using the fact that the red and orange circles bound disks in $W_{124}$, while the other $1$--handles from $W_{234}$ are preserved. Similarly a cut-system for $W_{134}$, in orange, is deduced from the cut-system for $W_{123}$ and the blue and purple circles. To avoid confusion, Figure~\ref{figDiagramRP5} only shows two orange curves; the other can be obtained by sliding blue curves along themselves in order to get curves disjoint from the first two orange curves.
For the curves which come from $2$--handles on the Kirby diagram, beware that the framing induced by the surface must coincide with the framing indicated in Figure~\ref{figKirbyX}.
\end{example}

\section{Homology from a multisection diagram}
\label{sec:homology}

Throughout this section, we consider homology with $\Z$--coefficients.

We fix an $n$--sected $(n+1)$--manifold $W=\cup_{i=1}^n W_i$ with associated diagram $(\Sigma;\alpha^1,\dots,\alpha^n)$.
For $1\leq i\leq n$, $L_i$ is the subgroup of $H_1(\Sigma)$ generated by the homology classes of the $\alpha^i$--curves.
Note that $L_i$ is the kernel of the inclusion map $H_1(\Sigma)\to H_1(W_I)$ for $I=\{1,\dots,n\}\setminus\{i\}$.

For $1<\ell\leq n$, define a map $\delta_l:\bigoplus_{|I|=\ell}\left(\cap_{i\in I}L_i\right)\to\bigoplus_{|I|=\ell-1}\left(\cap_{i\in I}L_i\right)$ as follows: for $c\in\cap_{i\in I}L_{i}$, the coordinate of $\delta_\ell(c)$ in $\cap_{\substack{i\in I \\i\neq j}}L_{i}$ is $(-1)^{|\{s\in I\mid s<j\}|}c$. Also define a map $\delta_1:\bigoplus_{i=1}^nL_i\to H_1(\Sigma)$ given by the inclusions.

\begin{theorem} \label{th:homology}
 The homology of $W$ is the homology of the complex:
 $$0\to\Z\xrightarrow{0}\bigoplus_{i=1}^n\left(\cap_{j\neq i}L_j\right)\xrightarrow{\delta_{n-1}}\dots\xrightarrow{\delta_{\ell+1}}\bigoplus_{|I|=\ell}\left(\cap_{i\in I}L_i\right)\xrightarrow{\delta_\ell}\dots\xrightarrow{\delta_2}\bigoplus_{i=1}^nL_i\xrightarrow{\delta_1} H_1(\Sigma)\xrightarrow{0}\Z\to 0.$$
 In particular, $H_1(W)\cong H_1(\Sigma)/\oplus_{i=1}^nL_i$ and $H_n(W)\cong\cap_{i=1}^nL_i$.
\end{theorem}
This result will follow from Propositions~\ref{prop:homologySpines} and~\ref{prop:IsomComplexes}.

For $2\leq k\leq n+1$, we denote by $Y_k$ the $k$--spine of the multisection; in particular $Y_{n+1}=W$ and $Y_2=\Sigma$.
First note that $W$ can be constructed from $Y_k$ by adding cells of dimension at least $k$, so that $H_k(W)=H_k(Y_\ell)$ for any $\ell\geq k+2$. The strategy is to compute the homology of $Y_k$ by induction on $k$, using the long exact sequence of the pair $(Y_{k+1},Y_k)$. This provides a complex which gives the homology of $W$, similar to the complex associated to a CW--decomposition.

To compute the homology of the pair $(Y_{k+1},Y_k)$, we need to understand the homology rel boundary of $1$--handlebodies.

\begin{lemma} \label{lem:homologyHandlebody}
 Let $V$ be a $(k+1)$--dimensional $1$--handlebody of genus $g$. The homology of $(V,\partial V)$ is given by
 $$H_\ell(V,\partial V)\cong\left\lbrace\begin{array}{ll} \Z & \text{if } \ell=k+1 \\ \Z^g & \text{if } \ell=k \\ 0 & \text{otherwise} \end{array}\right..$$
\end{lemma}
\begin{proof}
 Since $V$ deformation-retracts on a bouquet of $g$ circles, and $\partial V\cong\sharp^g(S^1\times S^{k-1})$, we have $$H_\ell(V)\cong\left\lbrace\begin{array}{ll} \Z & \text{if } \ell=0 \\ \Z^g & \text{if } \ell=1 \\ 0 & \text{otherwise} \end{array}\right. \qquad\text{and}\qquad H_\ell(\partial V)\cong\left\lbrace\begin{array}{ll} \Z & \text{if } \ell=0,k \\ \Z^g & \text{if } \ell=1,k-1 \\ 0 & \text{otherwise} \end{array}\right..$$
 Conclude using the long exact sequence of the pair $(V,\partial V)$.
\end{proof}

\begin{corollary} \label{cor:homologyPairs}
 The homology group $H_\ell(Y_{k+1},Y_k)$ is free abelian and is non trivial if and only if $\ell=k,k+1$. Moreover, the rank of $H_{k+1}(Y_{k+1},Y_k)$ is $\binom{n}{n-k+1}$.
\end{corollary}
\begin{proof}
 These homology groups split as follows: $H_\ell(Y_{k+1},Y_k)=\oplus_{|I|=n-k+1}H_\ell(W_I,\partial W_I)$.
\end{proof}

\begin{proposition} \label{prop:homologySpines}
 Let $n\geq1$. For $2\leq k\leq n+1$, the homology of the $k$--spine $Y_k$ is the homology of the complex
 \begin{equation} \tag{$*$}
 0\to \Z^{d_k} \xrightarrow{0} H_{k-1}(Y_k,Y_{k-1}) \xrightarrow{\partial_{k-1}}\dots\xrightarrow{\partial_2} H_2(Y_3,Y_2) \xrightarrow{\partial_1} H_1(Y_2)\xrightarrow{0}\Z\to 0
 \end{equation}
 where $d_k=\binom{n-1}{k-2}$. Moreover, for $k\leq n$, the group $H_k(Y_k)$ is generated by the homology classes of the $\partial W_I$ for $|I|=n-k+1$.
\end{proposition}
\begin{proof}
 We prove the result by induction on $k$; it is immediate for $k=2$ (since $Y_2=\Sigma$).
 We use the long exact sequence of the pair $(Y_{k+1},Y_k)$. Thanks to Corollary~\ref{cor:homologyPairs}, it gives $H_\ell(Y_{k+1})=H_\ell(Y_k)$ when $\ell\leq k-2$. Further, by induction, the map  $H_{k+1}(Y_{k+1},Y_k) \to H_k(Y_k)$ is surjective, so that the left part of the long exact sequence splits into the following two sequences.
 \begin{eqnarray}
  & 0\to H_{k+1}(Y_{k+1}) \to H_{k+1}(Y_{k+1},Y_k) \to H_{k}(Y_k) \to0 \\
  & 0\to H_k(Y_{k+1}) \xrightarrow{g_k} H_k(Y_{k+1},Y_k) \xrightarrow{f_k} H_{k-1}(Y_k) \to H_{k-1}(Y_{k+1}) \to0
 \end{eqnarray}
 We shall see, using ($2$) and the analogous sequence for lower values of $k$, that the complex ($*$) gives $H_\ell(Y_{k+1})$ for $\ell=k-1,k$. The following diagram merges part of the complex ($*$) with parts of the sequences ($2$).
 \begin{center}
 \begin{tikzpicture}
 \node (a) at (2,1) {$H_{k-1}(Y_k)$};
 \node (b) at (0,0) {$H_k(Y_{k+1},Y_k)$};
 \node (c) at (4,0) {$H_{k-1}(Y_k,Y_{k-1})$};
 \node (d) at (8,0) {$H_{k-2}(Y_{k-1},Y_{k-2})$};
 \node (e) at (6,-1) {$H_{k-2}(Y_{k-1})$};
 \node (f) at (-2.2,0) {};
 \node (g) at (10.5,0) {};
 \draw[->,>=latex] (b) -- (a) node[pos=0.4,above left,black] {$\scriptstyle{f_k}$};
 \draw[->,>=latex] (c) -- (e) node[pos=0.4,below left,black] {$\scriptstyle{f_{k-1}}$};
 \draw[->,>=latex] (a) -- (c) node[pos=0.6,above right,black] {$\scriptstyle{g_{k-1}}$};
 \draw[->,>=latex] (e) -- (d) node[pos=0.6,below right,black] {$\scriptstyle{g_{k-2}}$};
 \draw[->,>=latex] (f) -- (b) node[pos=0.5,above,black] {$\scriptstyle{0}$};
 \draw[->,>=latex] (b) -- (c) node[pos=0.5,above,black] {$\scriptstyle{\partial_{k-1}}$};
 \draw[->,>=latex] (c) -- (d) node[pos=0.5,above,black] {$\scriptstyle{\partial_{k-2}}$};
 \draw[->,>=latex] (d) -- (g);
 \end{tikzpicture}
 \end{center}
 Note that the $g_\ell$ are injective. From ($2$) and the injectivity of $g_{k-1}$ we get $H_{k-1}(Y_{k+1})\cong H_{k-1}(Y_k)/\mathrm{Im}(f_k)\cong g_{k-1}(H_{k-1}(Y_k))/\mathrm{Im}(\partial_{k-1})$. Now $g_{k-1}(H_{k-1}(Y_k))=\ker(f_{k-1})=\ker(\partial_{k-2})$, so that $H_{k-1}(Y_{k+1})\cong\ker(\partial_{k-2})/\mathrm{Im}(\partial_{k-1})$ as desired. Also $H_k(Y_{k+1})\cong\mathrm{Im}(g_k)=\ker(f_k)=\ker(\partial_{k-1})$.

 By ($1$), the rank of $H_{k+1}(Y_{k+1})$ is $d_{k+1}=\binom{n}{n-k+1}-d_k=\binom{n-1}{k-1}$. For $I\in\{1,\dots,n\}$ such that $|I|=n-k$, $\partial W_I$ has a homology class in $H_{k+1}(Y_{k+1})$. The classes of the $\partial W_I$ for $I\in\{1,\dots,n-1\}$ are independent and there are $\binom{n-1}{n-k}=d_{k+1}$ such classes, so they generate $H_{k+1}(Y_{k+1})$.
\end{proof}

To deduce the theorem from Proposition~\ref{prop:homologySpines}, we need to express the $H_k(Y_{k+1},Y_k)$ in terms of the subgroups $L_i$ of $H_1(\Sigma)$. Note that, for $2\leq k\leq n$, $H_k(Y_{k+1},Y_k)\cong\oplus_{|I|=n-k+1}H_k(W_I,\partial W_I)$. The idea is to identify $H_k(W_I,\partial W_I)$ with $\cap_{i\notin I}L_i$ and to use these identifications to construct an isomorphism of complexes.
Recall $W=Y_{n+1}$ and $\Sigma=Y_2$.

\begin{proposition} \label{prop:IsomComplexes}
 Let $n>1$. There is an isomorphism of complexes
 \begin{center}
 \begin{tikzpicture}
 \node (a) at (0,1.5) {$0$};
 \node (b) at (1,1.5) {$\Z$};
 \node (c) at (3.2,1.5) {$H_n(Y_{n+1},Y_n)$};
 \node (d) at (6,1.5) {$\dots$};
 \node (e) at (8.5,1.5) {$H_2(Y_3,Y_2)$};
 \node (f) at (11,1.5) {$H_1(Y_2)$};
 \node (g) at (13,1.5) {$\Z$};
 \node (h) at (14,1.5) {$0$};
 \draw[->,>=latex] (a) -- (b) node[pos=0.5,above,black] {};
 \draw[->,>=latex] (b) -- (c) node[pos=0.5,above,black] {$\scriptstyle{0}$};
 \draw[->,>=latex] (c) -- (d) node[pos=0.5,above,black] {$\scriptstyle{\partial_{n-1}}$};
 \draw[->,>=latex] (d) -- (e) node[pos=0.5,above,black] {$\scriptstyle{\partial_2}$};
 \draw[->,>=latex] (e) -- (f) node[pos=0.5,above,black] {$\scriptstyle{\partial_1}$};
 \draw[->,>=latex] (f) -- (g) node[pos=0.5,above,black] {$\scriptstyle{0}$};
 \draw[->,>=latex] (g) -- (h) node[pos=0.5,above,black] {};
 \node (a') at (0,0) {$0$};
 \node (b') at (1,0) {$\Z$};
 \node (c') at (3.2,0) {$\bigoplus_{i=1}^n\left(\cap_{j\neq i}L_j\right)$};
 \node (d') at (6,0) {$\dots$};
 \node (e') at (8.5,0) {$\bigoplus_{i=1}^nL_i$};
 \node (f') at (11,0) {$H_1(\Sigma)$};
 \node (g') at (13,0) {$\Z$};
 \node (h') at (14,0) {$0$};
 \draw[->,>=latex] (a') -- (b') node[pos=0.5,above,black] {};
 \draw[->,>=latex] (b') -- (c') node[pos=0.5,above,black] {$\scriptstyle{0}$};
 \draw[->,>=latex] (c') -- (d') node[pos=0.5,above,black] {$\scriptstyle{\delta_{n-1}}$};
 \draw[->,>=latex] (d') -- (e') node[pos=0.5,above,black] {$\scriptstyle{\delta_2}$};
 \draw[->,>=latex] (e') -- (f') node[pos=0.5,above,black] {$\scriptstyle{\delta_1}$};
 \draw[->,>=latex] (f') -- (g') node[pos=0.5,above,black] {$\scriptstyle{0}$};
 \draw[->,>=latex] (g') -- (h') node[pos=0.5,above,black] {};
 \draw[->,>=latex] (b) -- (b') node[pos=0.5,right,black] {$\scriptstyle{\mathrm{Id}}$};
 \draw[->,>=latex] (c) -- (c') node[pos=0.5,right,black] {$\scriptstyle{h_n}$};
 \draw[->,>=latex] (e) -- (e') node[pos=0.5,right,black] {$\scriptstyle{h_2}$};
 \draw[->,>=latex] (f) -- (f') node[pos=0.5,right,black] {$\scriptstyle{h_1}$};
 \draw[->,>=latex] (g) -- (g') node[pos=0.5,right,black] {$\scriptstyle{\mathrm{Id}}$};
 \end{tikzpicture}
 \end{center}
 where $h_1$ is the identity and, for $k>1$, $h_k$ is the direct sum of isomorphisms $H_k(W_I,\partial W_I)\xrightarrow{\scriptstyle{\cong}}\cap_{i\notin I}L_i$.
 Further, there is an isomorphism $h:H_n(W)\to\cap_{i=1}^nL_i$ such that the following diagram commutes, where the top map is given by the long exact sequence of the pair $(W,Y_n)$.
 \begin{center}
 \begin{tikzpicture}
  \node (a) at (0,2) {$H_n(W)$};
  \node (b) at (4,2) {$H_n(Y_{n+1},Y_n)$};
  \node (c) at (0,0) {$\cap_{i=1}^nL_i$};
  \node (d) at (4,0) {$\bigoplus_{j=1}^n\left(\cap_{i\neq j}L_i\right)$};
  \draw[->,>=latex] (a) -- (b);
  \draw[->,>=latex] (a) -- (c) node[pos=0.5,left,black] {$\scriptstyle{h}$};
  \draw[->,>=latex] (b) -- (d) node[pos=0.5,right,black] {$\scriptstyle{h_n}$};
  \draw[->,>=latex] (c) -- (d) node[pos=0.5,above,black] {$\scriptstyle{\delta_n}$};
 \end{tikzpicture}
 \end{center}
\end{proposition}
\begin{proof}
 We prove the result by induction on the dimension $n+1$. Let us define $h_2$. We have $H_2(Y_3,Y_2)=\oplus_{|I|=n-1}H_2(W_I,\Sigma)$. These $W_I$ are $3$--dimensional handlebodies, so the exact sequence of the pair $(W_I,\Sigma)$ gives the short exact sequence $$0\to H_2(W_I,\Sigma)\to H_1(\Sigma)\to H_1(W_I)\to 0.$$ Thus we have a natural identification $H_2(W_I,\Sigma)\cong\ker\big(H_1(\Sigma)\to H_1(W_I)\big)=L_i$, where $I=\{1,\dots,n\}\setminus\{i\}$. Combining these identifications gives an isomorphism $h_2:H_2(Y_3,Y_2)\to \oplus_{i=1}^nL_i$ such that $\delta_1\circ h_2=\partial_1$.

 Now we take $k>2$ and we define $h_k$. We have $H_k(Y_{k+1},Y_k)=\oplus_{|I|=n-k+1}H_k(W_I,\partial W_I)$. Since $W_I$ deformation-retracts on a bouquet of circles, the exact sequence of the pair $(W_I,\partial W_I)$ gives an isomorphism $H_k(W_I,\partial W_I)\cong H_{k-1}(\partial W_I)$. Consider the multisection $\partial W_I=\cup_{i\notin I}(W_I\cap W_i)$. By induction, there is an isomorphism $$H_{k-1}(\partial W_I)\xrightarrow{\scriptstyle{\cong}}\cap_{i\notin I}L_i$$ such that the following diagram commutes.
 \begin{center}
 \begin{tikzpicture}
  \node (a) at (0,2.18) {$H_{k-1}(\partial W_I)$};
  \node (b') at (8,2.18) {\phantom{$\displaystyle H_{k-1}(\partial W_I,Y_{k-1}\cap\partial W_I)=\bigoplus_{|J|=n-k+2,J\supset I}H_{k-1}(W_J,\partial W_J)$}};
  \node (b) at (8,2) {$\displaystyle H_{k-1}(\partial W_I,Y_{k-1}\cap\partial W_I)=\bigoplus_{|J|=n-k+2,J\supset I}H_{k-1}(W_J,\partial W_J)$};
  \node (c) at (0,0) {$\cap_{i\notin I}L_i$};
  \node (d) at (8,0) {$\bigoplus_{j\notin I}\left(\cap_{i\notin I\cup\{j\}}L_i\right)$};
  \draw[->,>=latex] (a) -- (b');
  \draw[->,>=latex] (a) -- (c) node[pos=0.5,left,black] {$\scriptstyle{\cong}$};
  \draw[->,>=latex] (b) -- (d) node[pos=0.5,right,black] {$\scriptstyle{h_{k-1}}$};
  \draw[->,>=latex] (c) -- (d) node[pos=0.5,above,black] {$\scriptstyle{\delta_{k-1}}$};
 \end{tikzpicture}
 \end{center}
 Combine these isomorphisms to define $h_k$.

 A direct computation shows that $\delta_n:\cap_{i=1}^nL_i\to\ker(\delta_{n-1})$ is an isomorphism. This provides the isomorphism $h$ since $h_n$ induces an isomorphism $\ker(\delta_{n-1})\cong\ker(\partial_{n-1})$.
\end{proof}

\begin{remark}
The isomorphisms $h_k$ from Proposition~\ref{prop:IsomComplexes} are in fact unique provided $h_1=\id$, $h_2$ is the natural identification $$L_i = H_2(W_{\{1,\dots,n\}\setminus\{i\}},\del W_{\{1,\dots,n\}\setminus\{i\}})\simeq \ker(H_1(\Sigma)\to H_1(W_{\{1,\dots,n\}\setminus\{i\}})$$ and all $h_k$ respect the natural decompositions into direct summands indexed by $I$ with $|I|=n-k+1$. Indeed $\delta_{k-1}$ is injective on each summand.
\end{remark}

\begin{remark}
Observe that the signs $(-1)^{|\{s\in I\mid s<j\}|}$ appearing in the definition of the differentials $\delta_l$
coincide with the signs appearing in the orientation convention of a multisected manifold in Remark~\ref{rem:orient}.
\end{remark}

To illustrate Theorem~\ref{th:homology}, we recover the homology of some $5$--manifolds from the diagrams of the previous section.

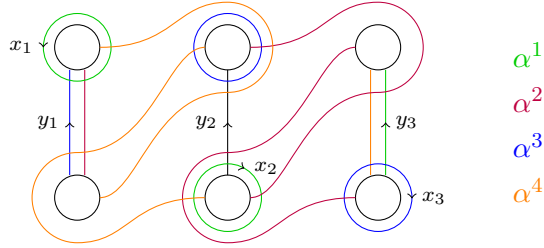
\begin{figure}[htb]
\begin{center}
\begin{tikzpicture}
 \foreach \x in {-2,0,2} \foreach \y in {-1,1}
 \draw (\x,\y) circle (0.3);
 \foreach \x/\c in {-2.1/blue,-1.9/purple,1.9/orange,2.1/vert,0/black}
 \draw[\c] (\x,-0.7) -- (\x,0.7);
 \foreach \x/\y/\c in {-2/1/vert,0/-1/vert,0/1/blue,2/-1/blue}
 \draw[\c] (\x,\y) circle (0.45);
 \foreach \t/\c in {0/purple,180/orange} {
 \draw[\c,rotate=\t] (0.3,1) .. controls +(1,0) and +(-0.7,0) .. (2,1.6) arc (90:-90:0.6) .. controls +(-1,0) and +(0.5,0) .. (0.3,-1);
 \draw[\c,rotate=\t] (1.7,-1) .. controls +(-1,0) and +(0.7,0) .. (0,-1.6) arc (270:90:0.6) .. controls +(1,0) and +(-0.5,0) .. (1.7,1);}
 \foreach \x/\p/\i in {-2.1/left/1,0/left/2,2.1/right/3}
 \draw[->] (\x,-0.01) -- (\x,0) node[\p] {$\scriptstyle{y_\i}$};
 \draw[->] (-2.45,1.01) -- (-2.45,1) node[left] {$\scriptstyle{x_1}$};
 \draw[->] (0.2,-0.6) -- (0.22,-0.61) node[right] {$\scriptstyle{x_2}$};
 \draw[->] (2.45,-1) -- (2.45,-1.01) node[right] {$\scriptstyle{x_3}$};
 \foreach \i/\c/\y in {1/vert/0.9,2/purple/0.3,3/blue/-0.3,4/orange/-0.9}
 \draw[\c] (4,\y) node {$\alpha^\i$};
\end{tikzpicture}
\caption{Quadrisection diagram of $S^2\times S^3$}
\label{fig:S2xS3ter}
\end{center}
\end{figure}

\stepcounter{theorem}

\subsection*{Example {\thetheorem}}
We start with $S^2\times S^3$, whose quadrisection diagram is reproduced in Figure~\ref{fig:S2xS3ter}, where some notations have been added for the computation. The homology classes of the curves $x_i$ and $y_i$ generate $H_1(\Sigma)$, and, denoting in the same way the curves and their homology classes, we have the following subgroups.
$$L_1=\langle x_1,x_2,y_3\rangle \quad
L_2=\langle y_1,x_3+y_2,x_2+y_3\rangle \quad
L_3=\langle y_1,x_2,x_3\rangle \quad
L_4=\langle x_1+y_2,x_2+y_1,y_3\rangle$$
$$L_{12}=\langle x_2+y_3 \rangle \quad
L_{13}=\langle x_2 \rangle \quad
L_{14}=\langle y_3 \rangle \quad
L_{23}=\langle y_1 \rangle \quad
L_{24}=\langle x_2+y_1+y_3 \rangle \quad
L_{34}=\langle x_2+y_1 \rangle$$
$$L_{123}=L_{124}=L_{134}=L_{234}=L_{1234}=0$$\newcommand{\congv}{
\begin{tikzpicture}
 \draw (0,0) node[rotate=-90] {$\cong$};
\end{tikzpicture}}
This gives the following complex:
\[ \xymatrix{
    0\ \ar[r] & \ \Z\ \ar[r] & \ 0\  \ar[r] & \ \oplus_{i<j}L_{ij}\ \ar[r]^{\partial_2} \ar@[white][d]^{\congv} & \ \oplus_i L_i\ \ar[r]^{\partial_1} \ar@[white][d]^{\congv} & \ H_1(\Sigma)\ \ar[r]^{\quad 0} \ar@[white][d]^{\congv} & \Z \ar[r] & \ 0 \\
    &&&\qquad  \Z^6 &\qquad  \Z^{12} &\qquad  \Z^6 &&
  }.\]
The map $\partial_1$ is surjective so $H_1=0$. We can check that $H_2\simeq \Z$ and is generated by the class of $(x_1,y_2+x_3,-x_3,-y_2-x_1)\in L_1\oplus L_2\oplus L_3\oplus L_4$. Finally $H_3\simeq \Z$ is the kernel of $\partial_2$ and is generated by $(x_2+y_3,-x_2,-y_3,-y_1,x_2+y_1+y_3,-x_2-y_1)\in L_{12}\oplus L_{13}\oplus L_{14}\oplus L_{23}\oplus L_{24}\oplus L_{34}$.
So we indeed recover the homology of $S^2\times S^3$.

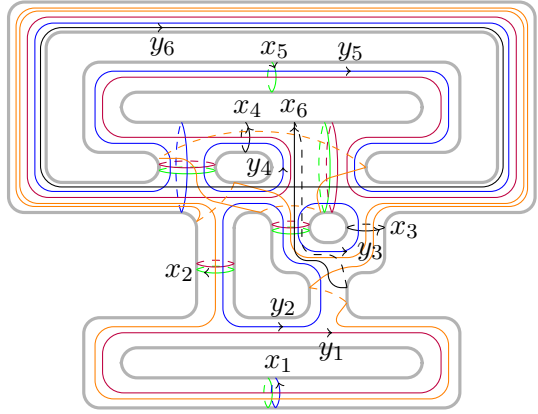
\begin{figure}[htb]
\begin{center}
\begin{tikzpicture} [xscale=0.7,yscale=0.5]
 \draw[gray!60,very thick,rounded corners=6pt] (4,-2) -- (0,-2) -- (0,-1) -- (8,-1) -- (8,-2) -- (4,-2)
 (4,-3) -- (-1,-3) -- (-1,0) -- (2,0) -- (2,3.5) -- (-3,3.5) -- (-3,10.5) -- (4,10.5)
 (3,3) -- (3,0) -- (5,0) -- (5,1.5) -- (4,1.5) -- (4,3.5) -- (3,3.5) -- (3,3)
 (4,-3) -- (9,-3) -- (9,0) -- (6,0) -- (6,1.5) -- (7,1.5) -- (7,3.5) -- (11,3.5) -- (11,10.5) -- (4,10.5)
 (4,7.5) -- (0,7.5) -- (0,6.5) -- (8,6.5) -- (8,7.5) -- (4,7.5)
 (4,9.5) -- (-2,9.5) -- (-2,4.5) -- (1,4.5) -- (1,5.5) -- (-1,5.5) -- (-1,8.5) -- (4,8.5)
 (3.5,4.5) -- (4,4.5) -- (4,5.5) -- (2.5,5.5) -- (2.5,4.5) -- (3.5,4.5)
 (4,9.5) -- (10,9.5) -- (10,4.5) -- (6.5,4.5) -- (6.5,5.5) -- (9,5.5) -- (9,8.5) -- (4,8.5)
 (5,3) -- (5,2.5) -- (6,2.5) -- (6,3.5) -- (5,3.5) -- (5,3);
 \foreach \x/\y/\c in {2/1.6/green,2/1.8/purple,4/2.9/green,4/3.1/purple,6/3/black} {
 \draw[\c]
 (\x,\y) .. controls +(0.2,-0.15) and +(-0.2,-0.15) .. (\x+1,\y);
 \draw[\c,dashed]
 (\x,\y) .. controls +(0.2,0.15) and +(-0.2,0.15) .. (\x+1,\y);
 }
 \foreach \x/\y/\c in {1/4.9/green,1/5.1/purple} {
 \draw[\c]
 (\x,\y) .. controls +(0.2,-0.15) and +(-0.2,-0.15) .. (\x+1.5,\y);
 \draw[\c,dashed]
 (\x,\y) .. controls +(0.2,0.15) and +(-0.2,0.15) .. (\x+1.5,\y);
 }
 \foreach \x/\y/\c in {3.9/-3/green,4.1/-3/blue,4/7.5/green,3.3/5.5/black} {
 \draw[\c]
 (\x,\y) .. controls +(0.15,0.2) and +(0.15,-0.2) .. (\x,\y+1);
 \draw[\c,dashed]
 (\x,\y) .. controls +(-0.15,0.2) and +(-0.15,-0.2) .. (\x,\y+1);
 }
 \foreach \x/\y/\c in {5.4/3.5/green,5.6/3.5/purple,1.6/3.5/blue} {
 \draw[\c]
 (\x,\y) .. controls +(0.2,0.5) and +(0.2,-0.5) .. (\x,\y+3);
 \draw[\c,dashed]
 (\x,\y) .. controls +(-0.2,0.5) and +(-0.2,-0.5) .. (\x,\y+3);
 }
 \draw[purple,rounded corners=8pt]
 (4,-0.5) -- (8.5,-0.5) -- (8.5,-2.5) -- (-0.5,-2.5) -- (-0.5,-0.5) -- (4,-0.5)
 (4,10) -- (-2.5,10) -- (-2.5,4) -- (4.5,4) -- (4.5,6) -- (-0.5,6) -- (-0.5,8) -- (8.5,8) -- (8.5,6) -- (6,6) -- (6,4) -- (10.5,4) -- (10.5,10) -- (4,10);
 \draw[blue,rounded corners=7pt]
 (4,9.8) -- (-2.3,9.8) -- (-2.3,4.2) -- (1.3,4.2) -- (1.3,5.8) -- (-0.7,5.8) -- (-0.7,8.2) -- (8.7,8.2) -- (8.7,5.8) -- (6.2,5.8) -- (6.2,4.2) -- (10.3,4.2) -- (10.3,9.8) -- (4,9.8)
 (2.2,5) -- (2.2,4.2) -- (4.3,4.2) -- (4.3,5.8) -- (2.2,5.8) -- (2.2,5)
 (2.7,3) -- (2.7,-0.3) -- (5.3,-0.3) -- (5.3,1.8) -- (4.3,1.8) -- (4.3,3.8) -- (2.7,3.8) -- (2.7,3)
 (4.7,3) -- (4.7,2.2) -- (6.3,2.2) -- (6.3,3.8) -- (4.7,3.8) -- (4.7,3);
 \draw[orange,rounded corners=7pt]
 (6,0.5) -- (5.5,-0.3) -- (8.7,-0.3) -- (8.7,-2.7) -- (4,-2.7) -- (-0.7,-2.7) -- (-0.7,-0.3) -- (2.5,-0.3) -- (2.5,3.8) -- (-2.7,3.8) -- (-2.7,10.2) -- (10.7,10.2) -- (10.7,3.8) -- (6.5,3.8) -- (6.5,2) -- (4.5,2) -- (4.5,4) -- (3,4.5)
 (2,3.2) -- (2.3,3.7) -- (-2.8,3.7) -- (-2.8,10.3) -- (10.8,10.3) -- (10.8,3.7) -- (6.7,3.7) -- (6.7,1.7) -- (5.7,1.7) -- (5,1)
 (1,5.3) -- (2,5.3) -- (2,4) -- (3.7,3.5)
 (5.3,3.5) -- (5.1,4.5) -- (6.5,5) ;
 \draw[orange,dashed]
 (3,4.5) .. controls +(-0.1,-0.5) and +(0.4,0.3) .. (2,3.2)
 (5,1) .. controls +(0.5,-0.1) and +(-0.3,0.3) .. (6,0.5)
 (3.7,3.5) .. controls +(0.3,0.3) and +(-0.3,0.3) .. (5.3,3.5)
 (6.5,5) .. controls +(-1,1) and +(1,0) .. (3.5,6.2) .. controls +(-0.5,0) and +(1,0.8) .. (1,5.3) ;
 \draw[rounded corners=5pt] (6,1) -- (5.5,1) -- (5.5,1.9) -- (4.6,1.9) -- (4.6,6.5);
 \draw[dashed,rounded corners=5pt] (6,1) -- (5.8,2.1) -- (4.8,2.1) -- (4.8,6) -- (4.6,6.5);
 \draw[rounded corners=8pt]
 (4,9.65) -- (-2.15,9.65) -- (-2.15,4.35) -- (10.15,4.35) -- (10.15,9.65) -- (4,9.65);
 \draw[->] (4.19,-2.2) -- (4.18,-2.16) node[above] {$x_1$};
 \draw[->] (2.21,1.5) -- (2.2,1.5) node[left] {$x_2$};
 \draw[->] (6.85,2.92) -- (6.87,2.93) node[right] {$x_3$};
 \draw[->] (3.37,6.3) -- (3.36,6.33) node[above] {$x_4$};
 \draw[->] (4.05,8.4) -- (4.08,8.3) node[above] {$x_5$};
 \draw[->] (4.6,6.2) -- (4.6,6.35) node[above] {$x_6$};
 \draw[->] (5.5,-0.5) -- (5.6,-0.5) node[below] {$y_1$};
 \draw[->] (4.2,-0.3) -- (4.3,-0.3) node[above] {$y_2$};
 \draw[->] (5.9,2.2) -- (6,2.2) node[right] {$y_3$};
 \draw[->] (4.3,4.9) -- (4.3,5) node[left] {$y_4$};
 \draw[->] (6,8.2) -- (6.1,8.2) node[above] {$y_5$};
 \draw[->] (1,9.65) -- (1.1,9.65) node[below] {$y_6$};
\end{tikzpicture}
\caption{Quadrisection diagram of $\RP^5$} \label{figHomologyRP5}
\end{center}
\end{figure}

\stepcounter{theorem}

\subsection*{Example {\thetheorem}} We now consider $W=\RP^5$, with the diagram in Figure~\ref{figHomologyRP5} reproducing that of Figure~\ref{figDiagramRP5}, with a basis of $H_1(\Sigma)$ represented. With the same color code as in Figure~\ref{fig:S2xS3ter}, we have the following subgroups of $H_1(\Sigma)$.
$$L_1=\langle x_1,x_2,x_3,x_4,x_5,x_6 \rangle \qquad L_2=\langle x_2,x_3,x_4-x_5,x_6,y_1,y_4+y_5 \rangle \qquad L_3=\langle x_1,x_6,y_2,y_3,y_4,y_5 \rangle$$
Although we have not drawn a complete cut-system for $\alpha^4$, we can compute $L_4$. We know that it is a $\Z^6$--summand of $H_1(\Sigma)$. The two generators given by the orange curves on the diagram are $h_1=-x_2+x_3+y_4$ and $h_2=x_2+x_4-x_6-y_1+y_2+2y_3-2y_6$. Further, the subgroup $L_4\cap L_3$ is orthogonal to $h_1$ and $h_2$ with respect to the intersection form on $\Sigma$, and isomorphic to $\Z^4$ (since $W_{123}$ is obtained from the genus--$6$ handlebody $W_{124}$ by exactly two surgeries). Hence it remains to compute the orthogonal complement of $\langle h_1,h_2 \rangle$ in~$L_3$. We finally get:
$$L_4=\langle -x_2+x_3+y_4,x_2+x_4-x_6-y_1+y_2+2y_3-2y_6,x_1-y_4,x_6-2y_4,y_2+y_3-y_4,y_5 \rangle.$$
We obtain the following subgroups of $H_1(\Sigma)$.
$$L_{12}=\langle x_2,x_3,x_4-x_5,x_6 \rangle \quad
L_{13}=\langle x_1,x_6 \rangle \quad
L_{14}=\langle x_1-x_2+x_3,2x_1-x_6 \rangle \quad L_{23}=\langle x_6,y_4+y_5 \rangle$$
$$L_{24}=\langle -x_2+x_3+y_4+y_5,x_6-2y_4-2y_5 \rangle \quad
L_{34}=\langle x_1-y_4,x_6-2y_4,y_2+y_3-y_4,y_5 \rangle$$
$$L_{123}=\langle x_6 \rangle \quad
L_{124}=\langle -2x_2+2x_3+x_6 \rangle \quad
L_{134}=\langle 2x_1-x_6 \rangle \quad
L_{234}=\langle x_6-2y_4-2y_5 \rangle$$
It follows that the homology of $W$ is the homology of the following complex:
$$ \xymatrix{
    0\ \ar[r] & \ \Z\ \ar[r]^{0\qquad} & \ \oplus_{i<j<k}L_{ijk}\  \ar[r]^{\partial_3} \ar@[white][d]^{\congv} & \ \oplus_{i<j}L_{ij}\ \ar[r]^{\partial_2} \ar@[white][d]^{\congv} & \ \oplus_i L_i\ \ar[r]^{\partial_1} \ar@[white][d]^{\congv} & \ H_1(\Sigma)\ \ar[r]^{\quad 0} \ar@[white][d]^{\congv} & \Z \ar[r] & \ 0 \\
    && \qquad \Z^4 & \qquad \Z^{16} & \qquad \Z^{24} & \qquad \Z^{12} &&}.$$
We readily see that $L_{1234}=0$, so that $H_4(W)=0$, and $H_1(W)\cong H_1(\Sigma)/\mathrm{Im}(\partial_1)\cong\Z/2\Z$ (generated by the class of $y_6$).
Computing the image of $\partial_3$ and $\partial_2$ and the kernel of $\partial_2$ and $\partial_1$ shows that $H_2(W)=0$ and $H_3(W)\cong\Z/2\Z$ (generated by the class of $(x_3-x_2,x_1,-x_1-x_3,y_4+y_5, y_4+y_5-x_2-x_3,x_1-y_4-y_5)$), as expected.

\def\cprime{$'$}
\providecommand{\bysame}{\leavevmode ---\ }
\providecommand{\og}{``}
\providecommand{\fg}{''}
\providecommand{\smfandname}{\&}
\providecommand{\smfedsname}{\'eds.}
\providecommand{\smfedname}{\'ed.}
\providecommand{\smfmastersthesisname}{M\'emoire}
\providecommand{\smfphdthesisname}{Th\`ese}

\end{document}